\documentclass[journal]{IEEEtran}

\usepackage{subfigure}

\usepackage{graphicx}
\usepackage[ruled,vlined]{algorithm2e}
\usepackage{enumitem}
\usepackage{multirow}
\usepackage{xcolor}

\ifCLASSINFOpdf
\else
\fi
\usepackage{amsmath}

\usepackage{amssymb}
\usepackage{amsthm}

\newtheorem{lemma}{Lemma}
\newtheorem{proposition}{Proposition}
\newtheorem{definition}{Definition}
\newtheorem{corollary}{Corollary}
\newtheorem{assumption}{Assumption}

\begin{document}
\title{A Novel Riemannian Optimization Approach to the Radial Distribution Network Load Flow Problem}

\author{Majid~Heidarifar,~\IEEEmembership{Student Member,~IEEE,}
        Panagiotis~Andrianesis,~\IEEEmembership{Member,~IEEE,}
        and~Michael~Caramanis,~\IEEEmembership{Senior~Member,~IEEE}
        %and~% <-this % stops a space
\thanks{
M. Heidarifar, P. Andrianesis and M. Caramanis are with the Systems Eng. Div., Boston University, Boston,
MA, e-mails: \{mheidari, panosa, mcaraman\}@bu.edu.
Research partially supported by the Sloan Foundation under grant G-2017-9723 and NSF AitF grant 1733827.}% <-this % stops a space
}


\markboth{}%
{Shell \MakeLowercase{\textit{et al.}}: Bare Demo of IEEEtran.cls for IEEE Journals}



\maketitle

\begin{abstract}
In this paper, we formulate the Load Flow (LF) problem in radial electricity distribution networks as an unconstrained Riemannian optimization problem, consisting of two manifolds, and we consider alternative retractions and initialization options.
Our contribution is a novel LF solution method, which we show belongs to the family of Riemannian approximate Newton methods guaranteeing monotonic descent and local superlinear convergence rate.
To the best of our knowledge, this is the first exact LF solution method employing Riemannian optimization.
Extensive numerical comparisons on several test networks illustrate that the proposed method outperforms other Riemannian optimization methods (Gradient Descent, Newton's), and achieves comparable performance with the traditional Newton-Raphson method, albeit besting it by a guarantee to convergence.
We also consider an approximate LF solution obtained by the first iteration of the proposed method, and we show that it significantly outperforms other approximants in the LF literature.
Lastly, we derive an interesting comparison with the well-known Backward-Forward Sweep method.
\end{abstract}


% \fi
\IEEEpeerreviewmaketitle

\section{Introduction}
\IEEEPARstart{E}{lectricity} distribution networks are undergoing unprecedented challenges guided by the desire for highly efficient and reliable grids, characterized by high penetration of intermittent renewable energy sources, storage devices, flexible loads, and transportation electrification.
A fast, efficient, and scalable Load Flow (LF) solution method is the foundation in several recent works on real-time Optimal Power Flow (OPF) \cite{Gan_EtAl-IEEEJournal2016}, load-side primary frequency control \cite{Zhao_EtAl-AutomaticControl2014}, reactive power control \cite{Arnold_EtAl-PowerSys2016}-\cite{Bolognani_EtAl-AutomaticControl2015}, and robust state estimation \cite{Jin_EtAl-AutomaticControl2019}, which aim at facilitating the transition to an increasingly active, distributed and dynamic power system.

%\subsection{Background and Motivation}

Applying Kirchhoff's laws to a power network results in a set of nonlinear LF equations whose solution yields the steady-state grid condition and is obtained by numerical solution of non-linear systems.
In distribution networks, which are typically operated in a radial configuration, besides the widely applicable Newton's method ---\emph{a.k.a.} the Newton-Raphson method--- there exist several, customized, LF numerical methods, e.g., the Backward-Forward Sweep \cite{Shirmohammadi_EtAl-powersys1988}, the implicit Z-bus \cite{Chen_EtAl-powerdelivery1991}, the current injection method \cite{daCosta_EtAl-PowerSys1999}, the direct method \cite{Teng-PowerDelivery2003}, and the holomorphic embedding method \cite{Heidarifar_EtAl-PESGM2019}. 
Furthermore, the LF problem can also be posed as a convex optimization problem; see, e.g., the approach in \cite{Jabr-PowerSys2006} that models the LF problem in a radial network as a conic programming problem by relaxing (non-convex) equality constraints to (convex) inequalities.
Recently, convex relaxation techniques have been widely applied to the OPF problem \cite{Low-ControlNetworkSystems2014}.
Such approaches, however, do not guarantee that the relaxed solution, in general, satisfies the LF equations, i.e., the applicable laws of physics.

Interestingly, the original LF problem posed as an equality-constrained optimization problem in Euclidean space can be reformulated to an equivalent unconstrained optimization problem whose search space is a Riemannian \emph{manifold}, %\footnote{
%Roughly speaking, a manifold is a topological space that locally resembles a Euclidean space.}
which, in turn, can be solved with Riemannian optimization methods.
Riemannian optimization minimizes a real-valued function over a smooth manifold \cite{Abs_EtAl-Book2008},
and has recently received considerable attention with several applications in signal processing, machine learning, computer vision, numerical linear algebra, etc.
Relying on the lower dimension and the underlying geometric characteristics of a manifold, its efficiency exceeds that of Euclidean constrained optimization.
The application of Riemannian optimization to power systems has been limited, however, by the lack of a well-organized procedure to convert a constrained optimization problem whose search space forms a smooth manifold into a Riemannian optimization instance \cite{Douik_EtAl-Arxiv}.
More importantly, the feasible space in most problems encountered in engineering disciplines, e.g., the OPF problem, is a non-smooth manifold as there are extra constraints in addition to having to reside on a manifold.
This imposes additional challenges to the widespread application of Riemannian optimization.

%Nonetheless, treatments for this case can be found, for example, in \cite{Liu_EtAl-Springer2019} COMMENT: very general -- leaves open questions which we do not address....

%\subsection{Literature Review}
Traditional optimization methods have been recently extended to the case of Riemannian optimization, e.g., the unconstrained (Riemannian) Gradient Descent, Newton's, trust region and approximate Newton methods in \cite{Abs_EtAl-Book2008}, (Riemannian) Stochastic Gradient Descent in \cite{Bonnabel-AutomaticControl2013}, and the (Riemannian) consensus method \cite{Tron_EtAl-AutomaticControl2013}. 
Riemannian optimization has also found several applications in low-rank matrix completion \cite{Boumal_EtAl-Advances2011, Vandereycken-SIAM2013, Cambier_EtAl-SIAM2016}, dimension reduction for Independent Component Analysis \cite{Theis_EtAl-Springer2009}, tensor decomposition \cite{Sun_EtAl-Transactions2016}, online learning \cite{Shalit_EtAl-JMech2012},  port-Hamiltonian systems \cite{Sato-AutomaticControl2017, Sato_EtAl-AutomaticControl2018}, feedback particle filters \cite{Zhang_EtAl-AutomaticControl2018}, and unscented Kalman filters \cite{Menegaz_EtAl-AutomaticControl2019}.

In the power systems literature, \cite{Bolognani_EtAl-AllertonConf2015} was the first to introduce the notion of a Power Flow (PF) manifold presenting several LF approximations using the concept of a \textit{tangent space}.
The PF manifold describes the LF equations of a power network with a general (radial or meshed) topology, and the tangent space 
%COMMENT: tangent space vs HYPERPLANE
around a flat start solution, which is a point on the PF manifold, is presented as the best linear approximation to the LF equations.
The proposed linear approximant in \cite{Bolognani_EtAl-AllertonConf2015} reduces to a DC LF model assuming zero shunt admittances and purely inductive lines.
Employing the approximate LF technique developed in \cite{Bolognani_EtAl-AllertonConf2015}, \cite{Hauswirth_EtAl-PowerTech2017} presents an online OPF technique using a discrete-time projected gradient descent scheme.
A continuous-time counterpart of \cite{Hauswirth_EtAl-PowerTech2017} is proposed in \cite{Hauswirth_EtAl-AllertonConf2016}.
Note that the approach taken by \cite{Bolognani_EtAl-AllertonConf2015}--%, Hauswirth_EtAl-PowerTech2017, 
\cite{Hauswirth_EtAl-AllertonConf2016} does not include a mapping of the solution from the tangent space to the manifold ---\emph{a.k.a.} a \textit{retraction} in the Riemannian optimization literature--- thus yielding an approximate (or sub-optimal) solution, which does not in general satisfy the LF equations.

Since our focus is on radial distribution networks, we employ the well-known \emph{Branch Flow Model} (BFM) \cite{Baran_EtAl-PowerDelivery1989} ---\emph{a.k.a.} the DistFlow model--- which uses the voltage and current squared magnitudes; angles can be recovered by the LF solution.
The BFM has been recently included into an OPF setting \cite{Farivar_EtAl-PowerSys2013}, resulting in a non-convex optimization problem, due to a quadratic equality constraint, which, when relaxed to a convex inequality constraint, yields a Second Order Cone Programming (SOCP) problem.
However, this relaxation is exact only when certain conditions are met, \cite{Farivar_EtAl-SmartGridCommConf2011, Bose_EtAl-AllertonConf2011, Gan_EtAl-CDC2012, Lavaei_EtAl-PowerSys2014, Gan_EtAl-AutCont2015, Huang_EtAl-PowerSys2017, Nick_EtAl-AutCont2018}, and may occasionally provide solutions that do not satisfy the LF equations, hence physically meaningless.
A linear LF model originating from the BFM, namely the \emph{simplified} DistFlow model  ---\emph{a.k.a.} the LinDistFlow model--- which was proposed in \cite{Baran_EtAl-PowerDelivary1989CapSizing, Baran_EtAl-PowerDelivery1989_Reconfiguration}, has also been  employed in approximate OPF settings.
In fact, the LinDistFlow solution seems to improve the quality of the linear approximant in \cite{Bolognani_EtAl-AllertonConf2015}; the latter becomes equivalent with the former, assuming zero shunt admittances and using a nonlinear change of coordinates motivated by the fact that the basic LF equations are purely quadratic in the voltage magnitudes.

%\subsection{Contribution}
To the best of our knowledge, this paper is the first application of Riemannian optimization that yields an exact (not approximate) LF solution.
More specifically, our contributions are as follows.
First, we formulate the radial distribution network BFM as an unconstrained Riemannian optimization problem, for which we propose alternative manifolds, retractions, and initializations.
Second, we introduce an exact LF solution method, and show it belongs to the Riemannian approximate Newton methods, guarantees descent at each iteration and local superlinear convergence rate.
Third, we show through extensive numerical comparisons on several test networks that the proposed approximate Newton method outperforms in terms of computational effort other Riemannian optimization methods, namely the Riemannian Gradient Descent and the Riemannian Newton's methods, and achieves comparable performance with the traditional Newton-Raphson method.
Fourth, we illustrate that the first iteration of the proposed method ---considered as an approximate solution to the LF problem--- outperforms in terms of accuracy approximate solution methods in the LF literature (\cite{Bolognani_EtAl-AllertonConf2015} and the LinDistFlow solution).
Fifth, we present an interesting comparison with the Backward-Forward Sweep method, which shows that while both methods stay on the PF manifold at each iteration, they move along different directions.

%\subsection{Paper Organization}
The remainder of the paper is organized as follows.
In Section \ref{secRO}, we briefly review concepts from Riemannian optimization.
In Section \ref{SecLFasRO}, we formulate the BFM based LF problem as an unconstrained Riemannian optimization problem, and we propose alternative retractions and initializations.
In Section \ref{SecSearchMethods}, we introduce the proposed Riemannian approximate Newton method.
In Section \ref{SecSimRes}, we present extensive numerical comparisons on several test distribution networks.
Finally, in Section \ref{SecConc}, we conclude and provide directions for further research.
To improve paper readability, proofs are moved to an Appendix.

\section{Riemannian Optimization} \label{secRO}
In this section, we provide a brief overview of concepts and notation used in the context of smooth manifolds (Subsection \ref{Sectmanifolds}), we define the essential tool of \textit{retractions} (Subsection \ref{SectRetraction}), and we introduce Riemannian optimization methods (Subsection \ref{ROMethods}).
The interested reader is referred to \cite{Tu-Book2011} and \cite{Abs_EtAl-Book2008} for a detailed exposition of manifolds and Riemannian optimization, respectively.

\subsection{Smooth Manifolds}
\label{Sectmanifolds}

% A \textit{topological manifold} is a Hausdorff,
% second countable, and para-compact topological space that locally resembles the Euclidean space through the notion of charts.
% A \textit{chart} is a \textit{homeomorphism}, i.e., a one-to-one and onto mapping, from an open subset of the manifold to an open subset of Euclidean space $\mathbb{R}^{d}$, where $d$ is the manifold dimension.
% Clearly, for the definition to hold, every point in the manifold must belong to at least one chart.
% A homeomorphism between the range of any two charts, i.e., $\mathbb{R}^{d}$ spaces, when the chart domains overlap, defines a \textit{transition function}.
% A (smooth) \textit{differentiable manifold} is a topological manifold in which the transition functions are (infinitely) differentiable.

Consider the set described by $\mathcal{M} = \{\mathbf{x} \in \mathbb{R}^{n} | \mathbf{h}(\mathbf{x}) = \mathbf{0}\}$,
where $\mathbf{h}: \mathbb{R}^{n}\mapsto\mathbb{R}^{m}$ is a smooth map with $m \leq n$.
Then, $\mathcal{M}$ is a \textit{smooth manifold} of dimension $n-m$ of $\mathbb{R}^{n}$ \cite{Boumal_EtAl-IMA2018}.
A Riemannian optimization problem is described as 
$ \min_{\mathbf{x} \in \mathcal{M}}f(\mathbf{x})$, where $\mathbf{x}$ is a vector of (unknown) variables on the manifold, and $f(\mathbf{x}): \mathcal{M}\mapsto\mathbb{R}$ a smooth real-valued function.
Analogous to the concept of locally approximating a function by its derivative, the notion of \textit{tangent space}, $\mathcal{T}_{\mathbf{x}}\mathcal{M}$, is defined for every point $\mathbf{x} \in \mathcal{M}$ to locally approximate the manifold around $\mathbf{x}$.
$\mathcal{T}_{\mathbf{x}}\mathcal{M}$ is a vector space expressed by:
\begin{equation}
    \mathcal{T}_{\mathbf{x}}\mathcal{M} = \{\boldsymbol{\xi}\in\mathbb{R}^{n}: \textrm{D}\mathbf{h}(\mathbf{x})[\boldsymbol{\xi}]=\mathbf{0} \},
    \label{TanSpEq}
\end{equation}
where $\textrm{D}\mathbf{h}(\mathbf{x})$ denotes the differential of $\mathbf{h}$ at $\mathbf{x}$.
The notion of the tangent space generalizes the concept of the directional derivative as represented by the term $\textrm{D}\mathbf{h}(\mathbf{x})[\boldsymbol{\xi}]$, i.e., the derivative of $\mathbf{h}$ at $\mathbf{x}$ along the direction $\boldsymbol{\xi}$.
The point $\mathbf{x}$ is translated as the center or zero vector in $\mathcal{T}_{\mathbf{x}}\mathcal{M}$, and any $\boldsymbol{\xi}$ that satisfies \eqref{TanSpEq} is called a \textit{tangent vector} \cite{Abs_EtAl-Book2008}. 
The notions of direction and length in $\mathcal{T}_{\mathbf{x}}\mathcal{M}$ are introduced by a \textit{Riemannian metric} expressed by the classical dot product $\langle., .\rangle$, thus turning 
$\mathcal{M}$ into a \textit{Riemannian} manifold of the Euclidean space $\mathbb{R}^{n}$. 
%Henceforth, $\mathcal{M}$ shall denote a smooth \textit{Riemannian} manifold of the Euclidean space. 

Given a smooth real-valued function $f(\mathbf{x}): \mathcal{M}\mapsto\mathbb{R}$, the notion of \textit{Riemannian gradient} of the smooth mapping $f$ at $\mathbf{x}\in\mathcal{M}$ is the unique tangent vector denoted by $\textrm{grad}f(\mathbf{x}) \in \mathcal{T}_{\mathbf{x}}\mathcal{M}$ that satisfies \cite[Eq.~3.31]{Abs_EtAl-Book2008}:
\begin{equation}
    \langle\mathrm{grad} f(\mathbf{x}), \boldsymbol{\xi}\rangle = \textrm{D}f(\mathbf{x})[\boldsymbol{\xi}], \ \ \forall \boldsymbol{\xi} \in \mathcal{T}_\mathbf{x}\mathcal{M}. \nonumber
    \label{RiemGradDef}
\end{equation}
We define $\bar{f}(\mathbf{x}) := f(\mathbf{x}):\mathbb{R}^n\mapsto\mathbb{R}$, and denote the classical (Euclidean) gradient of $\bar{f}$ at $\mathbf{x}$ by $\textrm{Grad}\bar{f}(\mathbf{x})$. 
The Riemannian gradient of $f$ at $\mathbf{x}$ is defined as follows.
\begin{definition}
The Riemannian gradient of the smooth mapping $f$ at $\mathbf{x}\in\mathcal{M}$ is the orthogonal projection of $\textup{Grad}\bar{f}(\mathbf{x})$ to the tangent space, denoted by \cite[Eq.~3.37]{Abs_EtAl-Book2008}:
\begin{equation} \label{RiemGrad}
    \textup{grad}f(\mathbf{x}) = \mathbf{\Pi}_{\mathbf{x}} \textup{Grad}\bar{f}(\mathbf{x}),
\end{equation}
where $\mathbf{\Pi}_{\mathbf{x}}: \mathbb{R}^{n} \mapsto \mathcal{T}_{\mathbf{x}}\mathcal{M}$ is the orthogonal projection matrix given by
$ \mathbf{\Pi}_{\mathbf{x}} = \mathbf{I}_{n} - {\textup{D}\mathbf{h}(\mathbf{x})}^T \big( {\textsc{D}\mathbf{h}(\mathbf{x})} {\textup{D}\mathbf{h}(\mathbf{x})}^T\big)^{-1} {\textup{D}\mathbf{h}(\mathbf{x})}$,
with $\mathbf{I}_n$ the $n \times n$ identity matrix \cite{Strang-Book2006}.
\end{definition}

The notion of \textit{Riemannian Hessian} requires taking the derivative of the gradient, thus implying moving between tangent spaces, which is achieved by some \textit{affine connection}. 
Employing the \textit{Riemannian connection} on $\mathcal{M}$, denoted by $\nabla$, the Riemannian Hessian of $f$ at $\mathbf{x}$ is defined as follows.
\theoremstyle{definition}
\begin{definition}
\cite[Def.~5.5.1 and Eq.~5.15]{Abs_EtAl-Book2008} The Riemannian Hessian of the smooth mapping $f$ at $\mathbf{x}\in\mathcal{M}$ is the linear mapping $\textup{hess}f(\mathbf{x})$ of $\mathcal{T}_{\mathbf{x}}\mathcal{M}$ into itself defined as:
\begin{equation}
    \textup{hess}f(\mathbf{x})[\boldsymbol{\xi}] =
    \nabla_{\boldsymbol{\xi}} \textup{grad}f(\mathbf{x})=
    \mathbf{\Pi}_{\mathbf{\mathbf{x}}} \textup{Dgrad}f(\mathbf{x}) [\boldsymbol{\xi}],
    \label{Hessian}
\end{equation}
for all $\boldsymbol{\xi} \in \mathcal{T}_{\mathbf{x}}\mathcal{M}$, where $\textup{Dgrad}f(\mathbf{x}) [\boldsymbol{\xi}]$ denotes the directional derivative of the Riemannian gradient along $\boldsymbol{\xi}$.
\end{definition}

\subsection{Retractions}
\label{SectRetraction}
Riemannian optimization requires a mapping from a tangent vector (as we move along a suitable direction on the tangent space) to the manifold.
\emph{Retractions} are such tractable mappings \cite{Abs_EtAl-Book2008} defined as follows.
\begin{definition}
\label{DefRetraction}
\cite[Def.~4.1.1]{Abs_EtAl-Book2008} A retraction at a point $\mathbf{x}\in\mathcal{M}$ is a smooth mapping denoted by $\mathcal{R}_{\mathbf{x}}: \mathcal{T}_{\mathbf{x}}\mathcal{M}\mapsto\mathcal{M}$ that satisfies:
\begin{enumerate}
    \item $\mathcal{R}_{\mathbf{x}}(\mathbf{0}_\mathbf{x})=\mathbf{x}$, known as the \emph{centering} or the \emph{consistency} condition, indicating that $\mathbf{0}_\mathbf{x}$, i.e., the origin of $\mathcal{T}_{\mathbf{x}}\mathcal{M}$, must map to the tangent point $\mathbf{x}$.
    \item $\frac{\mathrm d}{\mathrm d t}\mathcal{R}_\mathbf{x}(t\boldsymbol{\xi})|_{t=0}=\boldsymbol{\xi}$, known as the \emph{local rigidity} condition, requiring the mapping to locally move towards the same tangent vector direction at least up to the first order.
\end{enumerate}
\end{definition}

In practice, retraction mappings are obtained by exploiting the geometric properties of the manifold while considering the computational complexity ---see e.g., Fig. \ref{fig:Retraction}.
\begin{figure}
  \centering
  \includegraphics[width=0.4\linewidth]{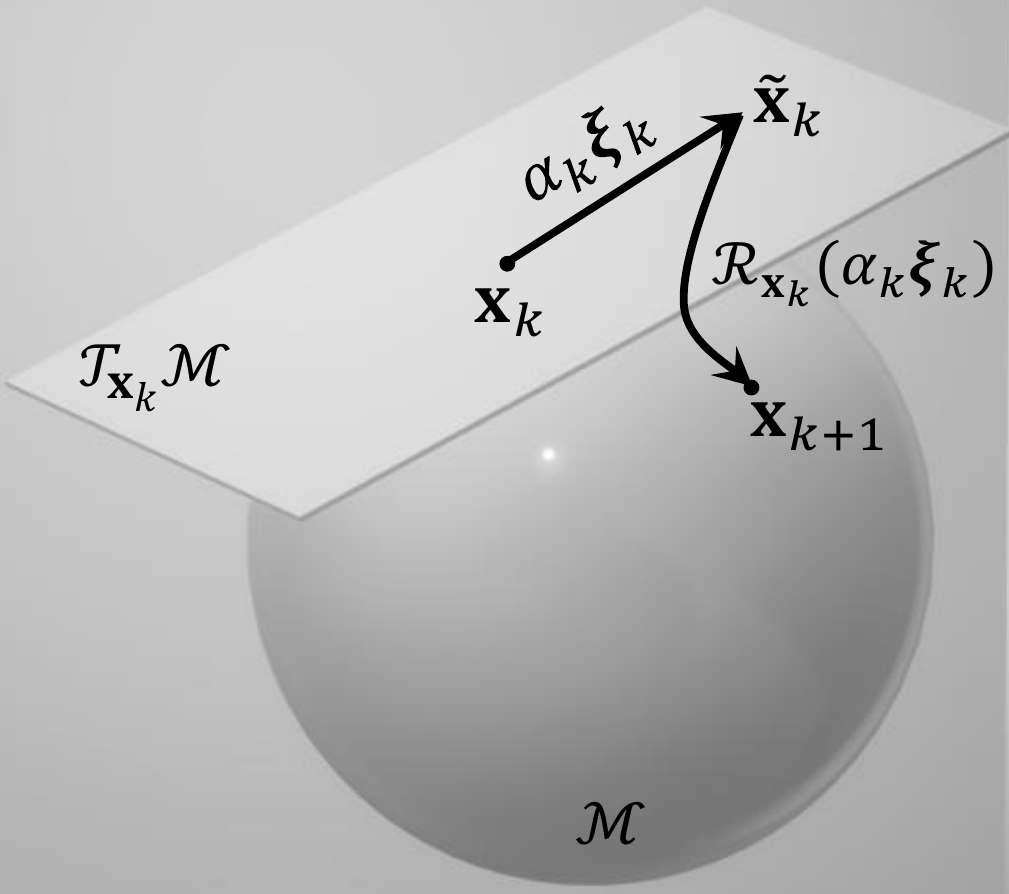}
  \caption{Retraction on the unit sphere. From point $\Tilde{\mathbf{x}}_k = \mathbf{x}_k+\alpha_k \boldsymbol{\xi}_k$, where $\alpha_k$ is a stepsize employed along direction $\boldsymbol{\xi}_k$, point $\mathbf{x}_{k+1}$ 
  on the manifold is obtained by normalization, i.e., $\mathbf{x}_{k+1} = \mathcal{R}_{\mathbf{x}_k}\big( \alpha_k \boldsymbol{\xi}_k\big) = \frac{\mathbf{x}_k +\alpha_k \boldsymbol{\xi}_k}{\|\mathbf{x}_k +\alpha_k \boldsymbol{\xi}_k\|_{2}}$ \cite{Abs_EtAl-Book2008}.
  }
  \label{fig:Retraction}
\end{figure}

\subsection{Riemannian Optimization Methods}
\label{ROMethods}
In Riemannian optimization, similar to classical unconstrained optimization, a descent search direction on the tangent space, $\boldsymbol{\xi} \in \mathcal{T}_\mathbf{x}\mathcal{M}$, satisfies the following condition:
\begin{equation}
    \langle\mathrm{grad} f(\mathbf{x}), \boldsymbol{\xi}\rangle < 0.
    \label{descentDirection}
\end{equation}
Several first and second order algorithms including Gradient Descent, Newton's and approximate Newton methods have been extended in \cite{Abs_EtAl-Book2008} to the Riemannian optimization settings.
Also, stepsize rules are directly applicable, e.g., the Armijo stepsize rule with scalars $\bar{\alpha} > 0, \beta,\sigma \in (0,1)$, which, at iteration $k$, finds the smallest non-negative integer $m$, satisfying:   \begin{equation} \label{Armijo}
    f(\mathbf{x}_{k}) - f\big(\mathcal{R}_{\mathbf{x}_{k}}(\beta^{m}\bar{\alpha}\boldsymbol{\xi}_{k})\big) \geq - \sigma \langle\mathrm{grad}f(\mathbf{x}_{k}), \beta^{m}\bar{\alpha}\boldsymbol{\xi}_{k} \rangle.  
  \end{equation}
In what follows, we provide a brief overview of Riemannian optimization methods.

\subsubsection{Riemannian Gradient Descent}
Algorithm \ref{GradDecAlg} reviews the Riemannian Gradient Descent (see \cite[Algorithm 1]{Abs_EtAl-Book2008}) using the Armijo stepsize rule.
It is shown in \cite[Thm. 4.3.1]{Abs_EtAl-Book2008} that Algorithm \ref{GradDecAlg} converges to a critical (stationary) point of $f$, with linear convergence rate \cite[Thm. 4.5.6 and Def. 4.5.1]{Abs_EtAl-Book2008}.
\begin{algorithm}[h]
\SetAlgoLined
\KwIn{$\mathcal{M}$, $f: \mathcal{M}\mapsto\mathbb{R}$,
$\mathcal{R}_\mathbf{x}: \mathcal{T}_\mathbf{x}\mathcal{M}\mapsto\mathcal{M}$, scalars $\bar{\alpha} > 0, \beta,\sigma \in (0,1)$, and $\epsilon > 0$ (small).}
\KwOut{Critical point $\mathbf{x}^{*}\in \mathcal{M}$ of $f : \mathcal{M}\mapsto\mathbb{R}$}
\textbf{Initialization:} $k=0$, $\mathbf{x}_{0} \in \mathcal{M}$\;
 \While{$\|\mathrm{grad}f(\mathbf{x}_{k})\|_{2} > \epsilon$}{
  Find direction $\boldsymbol{\xi}_k \in \mathcal{T}_{\mathbf{x}_{k}}\mathcal{M}$ satisfying \eqref{descentDirection}\;
  $\mathbf{x}_{k+1} = \mathcal{R}_{\mathbf{x}_{k}}(\beta^{m}\bar{\alpha}\boldsymbol{\xi}_{k})$ satisfying Armijo rule \eqref{Armijo}\;
  $k = k+1$\;
 }
 \caption{Riemannian Gradient Descent}
 \label{GradDecAlg}
\end{algorithm}

\subsubsection{Riemannian Newton's Method}
Algorithm \ref{NewtonAlg} reviews the Riemannian Newton's method (see \cite[Algorithm 5]{Abs_EtAl-Book2008}).
The direction $\boldsymbol{\xi}_k \in \mathcal{T}_{\mathbf{x}_{k}}\mathcal{M}$ is obtained by solving the Newton equation \eqref{NewtonEq}, where the Jacobian $\mathbf{J}( \mathbf{x}_k ) := \mathrm{hess} f ( \mathbf{x}_k )$,
  \begin{equation}
   \mathbf{J}( \mathbf{x}_k ) {\xi}_k =  \mathrm{hess} f(\mathbf{x}_k) [\boldsymbol{\xi}_k]
    = - \mathrm{grad} f(\mathbf{x}_k). \label{NewtonEq}
  \end{equation}
It is, however, not guaranteed that $\boldsymbol{\xi}_k$ is a descent direction unless $\mathrm{hess} f(\mathbf{x}_k)$ is positive definite.
Although the Riemannian Newton's method enjoys a local superlinear (at least quadratic) convergence rate \cite[Thm. 6.3.2]{Abs_EtAl-Book2008}, it lacks global convergence, i.e., there exist initial points for which the method does not converge \cite{Abs_EtAl-Book2008}.
In addition, evaluating the Hessian and solving \eqref{NewtonEq} may be computationally expensive.

\begin{algorithm}[h]
\SetAlgoLined
\KwIn{$\mathcal{M}$, $f : \mathcal{M}\mapsto\mathbb{R}$,  $\mathcal{R}_\mathbf{x}: \mathcal{T}_\mathbf{x}\mathcal{M}\mapsto\mathcal{M}$, scalar $\epsilon > 0 $ (small). }
\KwOut{Critical point $\mathbf{x}^{*}\in \mathcal{M}$ of $f : \mathcal{M}\mapsto\mathbb{R}$}
\textbf{Initialization:} $k=0$, $\mathbf{x}_{0} \in \mathcal{M}$\;
 \While{$\|\mathrm{grad}f(\mathbf{x}_{k})\|_2 > \epsilon$}{
  Find direction $\boldsymbol{\xi}_k \in \mathcal{T}_{\mathbf{x}_{k}}\mathcal{M}$ satisfying \eqref{NewtonEq}\;
  $\mathbf{x}_{k+1} = \mathcal{R}_{\mathbf{x}_{k}}(\boldsymbol{\xi}_{k})$\;
  $k = k+1$\;
 }
 \caption{Riemannian Newton's Method}
 \label{NewtonAlg}
\end{algorithm}

\subsubsection{Riemannian Approximate Newton Methods}
To overcome the drawbacks of Newton's method, \cite[Sec. 8.2]{Abs_EtAl-Book2008} presents approximate Newton methods, which maintain local superlinear convergence (under certain conditions), while having stronger global convergence properties, and require lower computational effort.
A class of approximate Newton methods approximates/modifies the Jacobian, so that \eqref{NewtonEq} becomes:
\begin{equation}
    \big[ \mathbf{J}( \mathbf{x}_k ) + \mathbf{E}_k \big] \boldsymbol{\xi}_k = - \mathrm{grad} f( \mathbf{x}_k ),
    \label{ApproximateJacob}
\end{equation}
where $\mathbf{E}_k$ denotes the approximation error, which is assumed to have sufficiently small bounds (in order to preserve the superlinear convergence).

\section{LF as a Riemannian Optimization Problem}
\label{SecLFasRO}
We consider a radial distribution network with node 0 representing the slack node, typically a distribution substation.
We denote the set of nodes, excluding the slack node, with $\mathcal{J} = \{ 1, 2, ..., J \}$.
Exploiting the radial topology, we denote the set of branches (i.e., lines) with $\mathcal{J}$, where branch $j$ has node $j$ as its downstream node, and we label its upstream node with $i$.
The set of branches whose upstream node is $j$ is denoted by $\mathcal{J'}(j)$.
At node $j$, $v_j$ is the squared voltage magnitude, $p_j$ and $q_j$ the net real and reactive power injections, with negative values representing power consumption.
The slack node voltage $v_0$ is typically assumed to be fixed.
At branch $j$, $P_{j}$ and $Q_{j}$ are the real and reactive power flows at the sending (upstream) end $i$, respectively, $l_j$ is the squared current magnitude, $r_{j}$ and $x_{j}$ the series resistance and reactance, and $a_j$ the transformer tap ratio \cite{Matpower2019}.
The total shunt admittance at node $j$ including the capacitance of the lines connected to node $j$ is denoted by $Y_j = G_j - \imath B_j$, where $\imath = \sqrt{-1}$.

In what follows, we present the BFM formulation (Subsection \ref{BFM}), the LF problem reformulation as a Riemannian optimization problem and the Riemannian gradient and Hessian calculations (Subsection \ref{RiemForm}), the proposed retractions (Subsection \ref{ProposedRetraction}), and initializations (Subsection \ref{PropInitial}).

\subsection{BFM Formulation} \label{BFM}
The LF equations of a radial network, using the BFM \cite{Baran_EtAl-PowerDelivery1989}, and accounting for shunt admittances and transformer tap ratios \cite{Matpower2019}, are given by:
\begin{gather}
    \sum_{j'\in \mathcal{J'}(j)}{P_{j'}} - P_{j} + {a_j^2} r_{j} {l_{j}}  + G_{j}v_{j}= p_j,  \ \ \ \forall j \in \mathcal{J}, \label{RealBalance}\\
    \sum_{j' \in \mathcal{J'}(j)}{Q_{j'}} - Q_{j} + {a_j^2}x_{j}{l_{j}} - B_{j}v_{j} = q_j,  \ \ \ \forall j \in \mathcal{J}, \label{ReactiveBalance}\\
    v_j = \frac{v_i}{a_j^2}  - 2(r_{j}P_{j}+x_{j}Q_{j}) + {a_j^2}(r_{j}^{2}+x_{j}^{2}){l_{j}}, \forall j \in \mathcal{J}, \label{VoltDrop}\\
    v_{i}l_{j} = P_{j}^{2} + Q_{j}^{2}, \ \ \ \forall j \in \mathcal{J}, \label{NonConvMan}
\end{gather}
where \eqref{RealBalance} and \eqref{ReactiveBalance} ensure real and reactive power balance, respectively, at node $j\in\mathcal{J}$, 
\eqref{VoltDrop} defines the voltage drop across line $j$, and 
\eqref{NonConvMan} describes the nonlinear relation between the current of line $j$, the real and reactive power flowing along line $j$ and the voltage at the upstream node $i$.

\begin{assumption} \label{Feasibility}
A feasible LF solution exists for the BFM described by \eqref{RealBalance}--\eqref{NonConvMan}.
\end{assumption}

Assumption \ref{Feasibility} is a very mild assumption, which generally holds for practical problems.
Problematic cases are discussed in our prior work \cite{Heidarifar_EtAl-PESGM2019}, which provides the means to diagnose them.
On the other hand, the BFM may generally admit multiple solutions, however, in practical networks \cite{Chiang_EtAl-CircSys1990}, with realistic resistance/reactance values and close to nominal substation voltage levels, the solution with practical voltage values is unique.
Notably, equations \eqref{RealBalance}--\eqref{VoltDrop} are linear, whereas \eqref{NonConvMan} represents the surface of a second order cone for each line $j\in\mathcal{J}$.
Relaxing \eqref{NonConvMan} to an inequality in the context of an optimization problem yields the aforementioned SOCP relaxation \cite{Farivar_EtAl-PowerSys2013}, whose solution, however, is not always guaranteed to be result in a binding inequality \cite{Farivar_EtAl-SmartGridCommConf2011}-\cite{Nick_EtAl-AutCont2018} and hence not satisfy \eqref{NonConvMan}.

\subsection{Riemannian Optimization Problem Formulations} \label{RiemForm}

In this subsection, we present two Riemannian optimization formulations for the LF problem, considering: \emph{(i)} a manifold represented by the full set of the LF equations \eqref{RealBalance}--\eqref{NonConvMan}, referred to as the BFM manifold, and \emph{(ii)} a manifold corresponding to \eqref{NonConvMan}, referred to as the Quadratic Equality (QE) manifold.

\subsubsection{BFM Manifold}

Consider the BFM, where we treat the real and reactive power injections, $p_j$ and $q_j$, as variables, and we add the following set of equations:
\begin{equation}
       p_j = \bar p_j, \qquad q_j = \bar q_j,   \qquad \forall j \in \mathcal{J}, \label{pqMismatch}\\
\end{equation}
with parameters $\bar p_j$ and $\bar q_j$ representing the values of the known injections.
Let $\mathbf{x}$ be the vector of variables, with $\mathbf{x} = \begin{pmatrix}
        \mathbf{u}^T \quad
        \mathbf{w}^T
    \end{pmatrix}^T$,
    and vectors $\mathbf{u}$ and $\mathbf{w}$ given by
    $\mathbf{u} =
    \begin{pmatrix}
        \mathbf{P}^T \,\, \mathbf{Q}^T \,\, \mathbf{l}^T \,\, \mathbf{v}^T 
    \end{pmatrix} ^T$, and  
    $\mathbf{w} = 
    \begin{pmatrix}
        \mathbf{p}^T \,\,
        \mathbf{q}^T
    \end{pmatrix} ^T$, respectively,
where $\mathbf{P}$, $\mathbf{Q}$, $\mathbf{l}$, $\mathbf{v}$, $\mathbf{p}$, and $\mathbf{q}$, are the vectors of variables $P_j$, $Q_j$, $l_j$, $v_j$, $p_j$, and $q_j$ respectively.
We define the BFM manifold as:
\begin{equation}
    \mathcal{M}_\textrm{BFM} = \{\mathbf{x} \in \mathbb{R}^{6J} |  \mathbf{F}_\textrm{BFM}(\mathbf{x})=\mathbf{0} \}, \label{BFM_Manifold}
\end{equation}
where $\mathbf{F}_\textrm{BFM}(\mathbf{x})=\mathbf{0}$ is the compact form of \eqref{RealBalance}-\eqref{NonConvMan}.
Naturally, the LF solution will be obtained when the values of variables $p_j$ and $q_j$ are equal to the values of the known injections, i.e., when \eqref{pqMismatch} holds.
Hence, the basic idea is to define an optimization problem, which penalizes the mismatches in \eqref{pqMismatch}, while ensuring that variables $\mathbf{x}$ remain on the BFM manifold.
This yields the following Riemannian optimization problem:
\begin{equation}
    \min_{\mathbf{x} \in \mathcal{M}_\textrm{BFM}} f_\textrm{BFM}(\mathbf{x})={\|\mathbf{w}-\mathbf{\bar w}\|}_2^2,
     \label{LF_as_Opt}
\end{equation}
where
    $\mathbf{ \bar w} = 
    \begin{pmatrix}
        \mathbf{\bar p}^T \,\,
        \mathbf{\bar q}^T
    \end{pmatrix} ^T$, 
is the vector of the known real and reactive power injections.
Given Assumption \ref{Feasibility}, the optimal solution of problem \eqref{LF_as_Opt} should be zero.
Denote by $\mathbf{x}_k$ the vector obtained at the $k$-th iteration.
Using \eqref{RiemGrad}, the Riemannian gradient associated with \eqref{LF_as_Opt} becomes:
\begin{equation}
   \textrm{grad}f_{\textrm{BFM}}(\mathbf{x}_{k}) = 2\mathbf{\Pi}_{\mathbf{x}_{k}}   \begin{pmatrix} \mathbf{0}_{4J\times1} \\ \mathbf{w}_k-\mathbf{\bar w}
   \end{pmatrix}.
    \label{OurRiemGradEntirePF}
\end{equation}
Using \eqref{Hessian} and the product rule for derivatives, the Riemannian Hessian is given by:
\begin{align}
    \textrm{hess}&f_{\textrm{BFM}}(\mathbf{x}_k)[\boldsymbol{\xi}_k]
    = \mathbf{\Pi}_{\mathbf{x}_k}
    \textrm{Dgrad}f_{\textrm{BFM}}(\mathbf{x}_k) [ \boldsymbol{\xi}_k],\nonumber \\
     &= \mathbf{\Pi}_{\mathbf{x}_k}  \left[2 \mathbf{\Pi}_{\mathbf{x}_k}  \begin{pmatrix}
    \mathbf{0}_{4J\times4J} \quad \mathbf{0}_{4J\times2J} \\
    \mathbf{0}_{2J\times4J} \quad \mathbf{I}_{2J\times2J} \\
    \end{pmatrix} + \mathbf{C}_{\mathbf{x}_k} \right] \boldsymbol{\xi}_k,
    \label{RiemHessProofEntirePFMan}
\end{align}
for all $\boldsymbol{\xi}_k \in \mathcal{T}_{\mathbf{x}_k}\mathcal{M}_\textrm{BFM}$, where $\mathbf{C}_{\mathbf{x}_k}$ is a matrix involving the derivatives of $\mathbf{\Pi}_{\mathbf{x}}$ at $\mathbf{x} = \mathbf{x}_k$.
It can be shown that the $n$-th column of $\mathbf{C}_{\mathbf{x}_k}$, denoted by $\mathbf{C}_{n,\mathbf{x}_k}$, equals:
\begin{equation}
    \mathbf{C}_{n, \mathbf{x}_k} = 2 \boldsymbol{\Gamma}_{n,\mathbf{x}_k}^{T}  \begin{pmatrix} \mathbf{0}_{4J\times1} \\ \mathbf{w}_k-\mathbf{\bar w}
   \end{pmatrix},
    \label{EqLixkEntirePFMan}
\end{equation}
where $\boldsymbol{\Gamma}_{n,\mathbf{x}_k}$ denotes the matrix to scalar derivative of $\mathbf{\Pi}_{\mathbf{x}}$ \emph{w.r.t.} the $n$-th element of $\mathbf{x}$ at $\mathbf{x} = \mathbf{x}_k$.

\subsubsection{QE Manifold}
Although the BFM manifold is a natural way to define the entire PF manifold in radial networks, inspired by the SOCP relaxation, we observe that \eqref{NonConvMan} can be written by completing the squares so as to resemble the equation of a sphere.
Hence, we define the QE manifold as follows:
\begin{equation}
    \mathcal{M}_\textrm{QE} = \{\mathbf{u} \in \mathbb{R}^{4J} |  \mathbf{F}_\textrm{QE}(\mathbf{u})=\mathbf{0} \}, \label{OurManDef}
\end{equation}
where $\mathbf{F}_\textrm{QE}(\mathbf{u})=0$ is the compact form of \eqref{NonConvMan}, and $\mathbf{u}$ is the vector defined above including variables $P_j$, $Q_j$, $l_j$, and $v_j$.
Notably, the QE manifold relies on vector $\mathbf{u}$ rather than the larger vector $\mathbf{x}$, because the real and reactive power injections, $\mathbf{p}$ and $\mathbf{q}$, respectively, are now considered parameters (instead of variables).
Then, considering the QE manifold, the LF solution requires that equations \eqref{RealBalance}--\eqref{VoltDrop} are satisfied.
Representing \eqref{RealBalance}--\eqref{VoltDrop} in a compact form as $\mathbf{Au}=\mathbf{b}$, where $\mathbf{A}$ is a $3J\times4J$ matrix, and $\mathbf{b}$ a $3J\times1$ vector that includes parameters $\mathbf{p}$ and $\mathbf{q}$, we define the following Riemannian optimization problem:
\begin{equation}
    \min_{\mathbf{u} \in  \mathcal{M}_\textrm{QE}} f_\textrm{QE}(\mathbf{u})=\|\mathbf{Au}-\mathbf{b}\|_{2}^{2},
    \label{OurRiemannOpt}
\end{equation}
whose objective function penalizes the mismatches in  \eqref{RealBalance}--\eqref{VoltDrop}.
Similarly to \eqref{LF_as_Opt}, given Assumption \ref{Feasibility}, the optimal solution of \eqref{OurRiemannOpt} should be zero.
The Riemannian gradient and Hessian associated with \eqref{OurRiemannOpt} are now given by:
\begin{align}
   \textrm{grad}f_\textrm{QE}(\mathbf{u}_{k}) &= 2\mathbf{\Pi}_{\mathbf{u}_{k}} \mathbf{A}^T (\mathbf{Au}_{k}-\mathbf{b}),
    \label{OurRiemGrad} \\
    \textrm{hess}f_\textrm{QE}(\mathbf{u}_k)[\boldsymbol{\zeta}_k]
    &= \mathbf{\Pi}_{\mathbf{u}_k}
    \textrm{Dgrad}f_\textrm{QE}(\mathbf{u}_k) [\boldsymbol{\zeta}_k],\nonumber \\
    & = \mathbf{\Pi}_{\mathbf{u}_k}\big(2 \mathbf{\Pi}_{\mathbf{u}_k} \mathbf{A}^{T}\mathbf{A} + \mathbf{{L}}_{\mathbf{u}_k} \big)\boldsymbol{\zeta}_k,
    \label{RiemHessProof}
\end{align}
where 
vector $\boldsymbol{\zeta}_k \in \mathbb{R}^{4J}$ is used instead of $\boldsymbol{\xi}_k \in \mathbb{R}^{6J}$ to avoid confusion (since we use variables $\mathbf{u}$ instead of $\mathbf{x}$), and $\mathbf{{L}}_{\mathbf{u}_k}$ is a matrix involving the derivative of $\mathbf{\Pi}_{\mathbf{u}}$ at $\mathbf{u} = \mathbf{u}_k$, whose $n$-th column, denoted by $\mathbf{{L}}_{n,\mathbf{u}_k}$, is expressed as:
\begin{equation}
    \mathbf{{L}}_{n, \mathbf{u}_k} = 2 \boldsymbol{\Lambda}_{n,\mathbf{u}_k}^{T}  \mathbf{A}^T (\mathbf{A} \mathbf{u}_{k}-\mathbf{b}),
    \label{EqLixk}
\end{equation}
where $\boldsymbol{\Lambda}_{n,\mathbf{u}_k}$ denotes the matrix to scalar derivative of $\mathbf{\Pi}_{\mathbf{u}}$ \emph{w.r.t.} the $n$-th element of $\mathbf{u}$ at $\mathbf{u} = \mathbf{u}_k$.

\subsection{Proposed Retractions} \label{ProposedRetraction}

In this subsection, we propose retraction methods that map a tangent vector to the BFM and the QE manifolds.
Recall that at the $k$-th iteration, using vector $\boldsymbol{\xi}_k$ without loss of generality, the retraction maps a point from the tangent space $\mathcal{T}_{\mathbf{x}_k}\mathcal{M}$, denoted by $\Tilde{\mathbf{x}}_k = \mathbf{x}_k + \alpha_k \boldsymbol{\xi}_k$, where $\alpha_k$ is the stepsize and $\boldsymbol{\xi}_k$ the search direction, to the manifold $\mathcal{M}$, to obtain the next point denoted by $\mathbf{x}_{k+1} = \mathcal{R}_{\mathbf{x}_k}(\alpha_{k}\boldsymbol{\xi}_k)$.
In what follows, we consider the variables associated with a single line, say line $j$.
To temporarily simplify notation, we drop the line subscripts of variables $P_j$, $Q_j$, and $l_j$ in \eqref{RealBalance}-\eqref{NonConvMan}, and we only show the iteration counter, i.e., $P_k$, $Q_k$, and $l_k$ for the $k$-th iteration.

\subsubsection{BFM Manifold Retraction}
By analogy to the Backward-Forward Sweep variant in \cite{Rajicic_EtAl-PowerDelivery1994}, retraction 
$\mathcal{R}^\textsc{BFM}$, 
involves a current update step followed by a voltage update step in a forward sweep manner starting from the root node, namely:
\begin{gather}
\underline{ \mathcal{R}^\textsc{BFM} } : \quad    P_{k+1} = 
     \Tilde{P}_{k}, \quad
     Q_{k+1} =
     \Tilde{Q}_{k},\quad
     l_{k+1} = 
     \frac{\Tilde{P}_{k}^2 + \Tilde{Q}_{k}^2}{{v}_{i,k+1}}, \quad 
      \nonumber \\
     v_{j, k+1} = 
     \frac{{v}_{i, k+1}}{a_j^2} 
     - 2(r\Tilde{P}_{k}+x\Tilde{Q}_{k}) 
     + a_j^2 (r^{2}+x^{2}){l}_{k+1}
     \label{RetBFM},
\end{gather}
with $\Tilde{P}_k$, $\Tilde{Q}_k$, $\Tilde{l}_k$, $\Tilde{v}_k$ denoting the values on the tangent space.
The retracted values for $p_{k+1}$ and $q_{k+1}$ are obtained by solving for the \emph{rhs} of \eqref{RealBalance} and \eqref{ReactiveBalance}, respectively, using the retracted values $P_{k+1}$, $Q_{k+1}$, $l_{k+1}$, and $v_{j,k+1}$ obtained in \eqref{RetBFM}.
\begin{lemma} \label{LemmaRetBFM}
$\mathcal{R}^\textsc{BFM}$ satisfies the conditions of Definition \ref{DefRetraction}.
\end{lemma}
The proof is included in the Appendix (Section \ref{AppRetBFM}).

\subsubsection{QE Manifold Retractions} 
We present two retractions that are inspired in part by the retraction on the unit sphere (Fig. \ref{fig:Retraction}) and the BFM geometry.
Inspired by the SOCP representation of \cite{Farivar_EtAl-PowerSys2013}, rearranging the completed square terms in \eqref{OurManDef}, we get 
%\begin{equation}
$     4{P_k}^{2} + 4{Q_k}^{2} + ({v_{i,k}}-{l_k})^2 = ({v_{i,k}}+{l_k})^2$,  %\nonumber
%\end{equation}
yielding: 
\begin{equation}
     \Big(\frac{2{P_k}}{{v_{i,k}}+{l_k}}\Big)^{2} + \Big(\frac{2{Q_k}}{{v_{i,k}}+{l_k}}\Big)^{2}+ \Big(\frac{{v_{i,k}}-{l_k}}{{v_{i,k}}+{l_k}}\Big)^{2} = 1. \label{SphereShapedMan}
\end{equation}
Notably, \eqref{SphereShapedMan} represents a sphere in $\mathbb{R}^3$ enabling retraction by normalization \cite{Abs_EtAl-Book2008}.
Retraction $\mathcal{R}^\textsc{QE}_{1}$ uses an identity mapping for $\Tilde{v}_{j,k}$ and normalizes each term in parentheses in \eqref{SphereShapedMan}, which after some algebra yields:
\begin{gather}
\underline{\mathcal{R}^\textsc{QE}_{1}} : \qquad 
    v_{j,k+1} = \Tilde{v}_{j,k}, \quad
    l_{k+1}=
    \frac{D_k+\Tilde{l}_k-\Tilde{v}_{i,k}}{D_k-\Tilde{l}_k+\Tilde{v}_{i,k}} \Tilde{v}_{i,k},  \qquad \qquad \nonumber \\
    P_{k+1}=
    \frac{2\Tilde{P}_k \Tilde{v}_{i,k}}{D_k-\Tilde{l}_k+\Tilde{v}_{i,k}}, \quad
    Q_{k+1}=
    \frac{2\Tilde{Q}_k \Tilde{v}_{i,k}}{D_k-\Tilde{l}_k+\Tilde{v}_{i,k}}, \label{RetQE1}
\end{gather}
where 
$D_{k} = \sqrt{4(\Tilde{P}_{k})^2+4(\Tilde{Q}_{k})^2+(\Tilde{l}_{k}-\Tilde{v}_{i,k})^2}$.
\begin{lemma}
\label{LemmaRetQE1}
$\mathcal{R}^\textsc{QE}_{1}$  
satisfies the conditions of Definition \ref{DefRetraction}.
\end{lemma}
The proof is included in the Appendix (Section \ref{AppRetQE1}).

Retraction $\mathcal{R}^\textsc{QE}_{2}$ uses identity mappings for $\Tilde{v}_{j,k}$, $\Tilde{P}_k$, and $\Tilde{Q}_k$ and updates the value of $l_{k+1}$ satisfying \eqref{OurManDef}; it is given by:
\begin{gather}
\underline{\mathcal{R}^\textsc{QE}_{2}} : \qquad \qquad
    v_{j,k+1} = 
     \Tilde{v}_{j,k},\quad
     l_{k+1} = 
     \frac{\Tilde{P}_{k}^2 + \Tilde{Q}_{k}^2}{\Tilde{v}_{i,k}}, \qquad \quad \qquad \nonumber \\
     P_{k+1} = 
     \Tilde{P}_{k}, \quad
     Q_{k+1} =
     \Tilde{Q}_{k}.\label{RetQE2}
\end{gather}
\begin{lemma} \label{LemmaRetQE2}
$\mathcal{R}^\textsc{QE}_{2}$ satisfies the conditions of Definition \ref{DefRetraction}.
\end{lemma}
The proof is included in the Appendix (Section \ref{AppRetQE2}).

\subsection{Proposed Initializations} \label{PropInitial}

The two proposed initializations are:

\textsc{Flat}: Flat start initialization in traditional LF methods (e.g., Newton-Raphson) sets $v_j$ equal to the substation voltage, and $l_j$, $P_j$ and $Q_j$ equal to zero.

\textsc{Warm}: Warm start initialization solves the LinDistFlow equations \cite{Baran_EtAl-PowerDelivary1989CapSizing, Baran_EtAl-PowerDelivery1989_Reconfiguration}: 
\begin{equation}
     \sum_{j'\in \mathcal{J'}(j)}{P_{j'}} - P_{j} + G_{j}v_{j}= p_j,  \ \ \ \forall j \in \mathcal{J}, \label{DistFlow1}   
\end{equation}
\begin{equation}
 \sum_{j' \in \mathcal{J'}(j)}{Q_{j'}} - Q_{j} - B_{j}v_{j} = q_j,  \ \ \ \forall j \in \mathcal{J}, \label{DistFlow2}   
\end{equation}
\begin{equation}
v_j - \frac{v_i}{a_j^2} + 2(r_{j}P_{j}+x_{j}Q_{j}) = 0, \forall j \in \mathcal{J}, \label{DistFlow3}    
\end{equation}    
i.e., it solves \eqref{RealBalance}--\eqref{VoltDrop}, assuming zero currents, to obtain the initial values for $v_j$, $P_j$ and $Q_j$.

For both \textsc{Flat} and \textsc{Warm}, the initial points, $\mathbf{x}_0$ and $\mathbf{u}_0$, for the BFM and QE manifolds, are obtained by applying retractions $\mathcal{R}^\textsc{BFM}$ and $\mathcal{R}^\textsc{QE}_{2}$, respectively.

\section{Proposed Approximate Newton Method}
\label{SecSearchMethods}

In this section, we present the proposed LF solution method, which is shown to belong to the category of Riemannian approximate Newton methods, and its application to the BFM and QE manifolds.
Without loss of generality, we present the method using the variables represented by vector $\mathbf{x}$, which refers to the BFM manifold, noting that vector $\mathbf{u}$, which refers to the QE manifold, is included in vector $\mathbf{x}$.  
Function $f(\mathbf{x})$ represents the mismatches during the optimization process.

Consider the $k$-th iteration, at which we are found at point $\mathbf{x}_k$ (on the manifold), which does not attain the minimum of the Riemannian optimization problem, i.e., $f(\mathbf{x}_k)$ is not zero; if it were, then we would have reached the LF solution, as $\mathbf{x}_k$ would be on the manifold and all equality constraints would be satisfied (zero mismatches).
We aim at finding a descent direction $\boldsymbol{\xi}_k$ on the tangent space $\mathcal{T}_{\mathbf{x}_k}\mathcal{M}$, to move from point $\mathbf{x}_k$ to point $\Tilde{\mathbf{x}}_k$, and then apply a retraction.
Instead of employing the Riemannian gradient, as in Algorithm \ref{GradDecAlg},
we obtain direction $\boldsymbol{\xi}_k$, so that the new point $\Tilde{\mathbf{x}}_k=\mathbf{x}_k+\boldsymbol{\xi}_k$ (assuming a stepsize equal to 1), which lies on the tangent space,
${\boldsymbol{\xi}_k \in  \mathcal{T}_{\mathbf{x}_k}\mathcal{M}}$, also minimizes the mismatches, i.e.,
\begin{equation}  \label{DirXi}
 f(\Tilde{\mathbf{x}}_k) = f(\mathbf{x}_k + \boldsymbol{\xi}_k) = 0   
\end{equation}
The LF solution algorithm employing the proposed approximate Newton method is presented in Algorithm \ref{ProposedAlg}.

\begin{algorithm}
\SetAlgoLined
\KwIn{$\mathcal{M}$, $f: \mathcal{M}\mapsto\mathbb{R}$, $\mathcal{R}_\mathbf{x}: \mathcal{T}_\mathbf{x}\mathcal{M} \mapsto\mathcal{M}$, scalars $\bar{\alpha} > 0, \beta,\sigma \in (0,1)$, and $\epsilon > 0$ (small).}
\KwOut{A critical point $\mathbf{x}^{*}\in \mathcal{M}$ of $f: \mathcal{M} \mapsto\mathbb{R}$}
\textbf{Initialization:} $k=0$ and $\mathbf{x}_{0} \in \mathcal{M}$ \;
 \While{$\|\mathrm{grad}f(\mathbf{x}_{k})\|_{2} > \epsilon$}
 {
  Find direction $\boldsymbol{\xi}_k \in  \mathcal{T}_{\mathbf{x}_k} \mathcal{M}$ satisfying \eqref{DirXi}\;
    $\mathbf{x}_{k+1} = \mathcal{R}_{\mathbf{x}_{k}}(\beta^{m}\bar{\alpha}\boldsymbol{\xi}_{k})$ satisfying Armijo rule \eqref{Armijo}\;
  $k = k+1$\;
 }
 \caption{Proposed Approximate Newton Method}
 \label{ProposedAlg}
\end{algorithm}

In what follows, we exemplify the algorithm on the BFM and the QE manifold.

\subsection{Application to the BFM Manifold}

Employing the BFM manifold, we obtain direction $\boldsymbol{\xi}_k = (\boldsymbol{\zeta}_k^T \  \boldsymbol{\eta}^T_k)^T$, where $\boldsymbol{\zeta}$ and $\boldsymbol{\eta}$ are search directions along $\textbf{u}$ and $\textbf{w}$ variables, respectively, by solving the following system of linear equations (with variables $\boldsymbol{\xi}_k$):
\begin{gather}
    \boldsymbol{\xi}_k \in \mathcal{T}_{\mathbf{x}_k}\mathcal{M}_\textrm{BFM}, \qquad
    \mathbf{w}_k+\boldsymbol{\eta}_k = \mathbf{\bar w}, \label{wmis}
\end{gather}
where the first set represents equations \eqref{RealBalance}--\eqref{VoltDrop} with variables $\boldsymbol{\xi}_k$ (instead of $\mathbf{x}_k$) and a linear approximation of \eqref{NonConvMan} ---which is presented in \eqref{TangCurIter} in the Appendix--- and the second set of equations is directly obtained by applying \eqref{DirXi} to $f_\textrm{BFM}(\mathbf{x}_k + \mathbf{\boldsymbol{\xi}_k}) = \|\mathbf{w}_k+\boldsymbol{\eta}_k - \mathbf{\bar w}\|_{2}^{2}$ given by \eqref{LF_as_Opt}.
Hence, one can think of the solution of \eqref{wmis}, assuming linearly independent rows, as the solution of minimizing $f_\textrm{BFM}$ over the tangent space around $\mathbf{x}_k$, expressed as     $\min_{\boldsymbol{\xi}_k \in  \mathcal{T}_{\mathbf{x}_k}\mathcal{M}_\textrm{BFM}} f_\textrm{BFM}(\mathbf{x}_k + \mathbf{\boldsymbol{\xi}_k})$,
where the decision variables are $\boldsymbol{\xi}_k$.

The following proposition summarizes the properties of the proposed LF solution method applied to the BFM manifold.
\begin{proposition}
\label{Prop1}
Algorithm \ref{ProposedAlg}, applied to the BFM manifold, is a Riemannian approximate Newton method, with guaranteed descent and local superlinear convergence rate.
\end{proposition}
The proof is included in the Appendix (Section \ref{ProofProp1}).

\subsection{Application to the QE manifold} 

Similarly to \eqref{wmis}, we obtain direction $\boldsymbol{\zeta}_k$ by solving the following system of linear equations:
\begin{equation}
\boldsymbol{\zeta}_k \in \mathcal{T}_{\mathbf{u}_k}\mathcal{M}_\textrm{QE}, \qquad
\mathbf{A}(\mathbf{u}_k + \boldsymbol{\zeta}_k) = \mathbf{b},
 \label{QE_xiTang}
\end{equation}
where the second set of equations is directly obtained by applying \eqref{DirXi} to $f_\textrm{QE}(\mathbf{u}_k + \mathbf{\boldsymbol{\zeta}_k}) = \|\mathbf{A}(\mathbf{u}_k + \boldsymbol{\zeta}_k) - \mathbf{b}\|_{2}^{2}$ given by \eqref{OurRiemannOpt}.
The solution of \eqref{QE_xiTang}, assuming linearly independent rows, can be viewed as the solution of minimizing $f_\textrm{QE}$ over the tangent space around $\mathbf{u}_k$, expressed as 
$\min_{\boldsymbol{\zeta}_k \in  \mathcal{T}_{\mathbf{u}_k}\mathcal{M}_\textrm{QE}} f_\textrm{QE}(\mathbf{u}_k + \mathbf{\boldsymbol{\zeta}_k})$,
where the decision variables are $\boldsymbol{\zeta}_k$. 
The following proposition essentially states that the properties of Proposition \ref{Prop1} carry over to the QE manifold.
\begin{proposition}
\label{Prop2}
Algorithm \ref{ProposedAlg}, applied to the QE manifold, is a Riemannian approximate Newton method, with guaranteed descent and local superlinear convergence rate.
\end{proposition}
The proof is included in the Appendix (Section \ref{ProofProp2}).

Although the proposed method is an exact LF solution method, executing it only for one iteration yields an approximate LF solution.
In fact, as we will show later, the numerical comparisons illustrate that the first iteration employing \textsc{Warm} and $\mathcal{R}^\textsc{QE}_{2}$ yields a higher quality LF solution compared with both LinDistFlow and the linear approximant proposed in \cite{Bolognani_EtAl-AllertonConf2015}. 
Lastly, the following Corollary relates the two initializations (\textsc{Flat} and \textsc{Warm}) when combined with retraction $\mathcal{R}^\textsc{QE}_{2}$. 
\begin{corollary} \label{cor1}
The first iteration of Algorithm \ref{ProposedAlg}, employing $\mathcal{R}^\textsc{QE}_{2}$ and \textsc{Flat}, yields the \textsc{Warm} initial point. 
\end{corollary}
The proof is included in the Appendix (Section \ref{Proofcor1}).

\section{Numerical Comparisons}
\label{SecSimRes}
In this section, we evaluate the performance of the Riemannian optimization methods ---namely the Riemannian Gradient Descent, the Riemannian Newton's and the proposed Riemannian approximate Newton methods--- on several standard IEEE radial distribution test networks with $18$, $22$, $33$, $69$, $85$, and $141$ nodes \cite{Zimmerman_EtAl-PowerSys2011}, as well as on single-phase equivalents of the IEEE-$13$, IEEE-$37$, and IEEE-$123$ test networks \cite{DistTestFeeders}.
We refer to the test networks as 13-node, 18-node, etc.
All methods are implemented in Matlab R2018a and tested on a desktop Intel i5-2500 at 3.3 GHz with 8 GB RAM. \footnote{The implementations are also made available in Manopt, a Matlab toolbox for optimization on manifolds \cite{manopt}.}
We added an additional stopping criterion in all Algorithms requiring the maximum voltage change in consecutive iterations be less than a small tolerance ($\delta > 0$), modifying the while-loop as follows: 
\begin{equation*}
\textbf{while } \|\mathrm{grad}f(\mathbf{x}_{k})\|_{2} > \epsilon \text{ or }  \|\sqrt{\mathbf{v}_{k+1}}-\sqrt{\mathbf{v}_{k}}\|_{\infty} > \ \delta, 
\end{equation*}
which better fits the LF problem and ensures a consistent comparison with the Newton-Raphson method. 
The tolerances are $\epsilon = \delta = 10^{-6}$.
Armijo parameters are set to $\beta=0.3$, $\sigma=0.05$ for all methods, $\bar{\alpha}=4.5$ for Riemannian Gradient on the BFM manifold, and $\bar{\alpha}=1$ for the other methods.
Computation times are reported in milliseconds (ms) and are obtained after running the main loop for $100$K times;
they do not include pre-processing or the initialization. 

The remainder of this section is structured as follows.
In Subsection \ref{SubRiemComp}, we compare the performance of the Riemannian optimization methods for the Base Case, which employs  \textsc{Warm} initialization, and, for the QE manifold, retraction $\mathcal{R}^\textsc{QE}_{2}$.
In Subsection \ref{SubCompInitRetr}, we evaluate the impact of \textsc{Flat} and $\mathcal{R}^\textsc{QE}_{1}$, as an alternative initialization and retraction, respectively.
In Subsection \ref{SubCompNR}, we compare the proposed method with the traditional Newton-Raphson method, considering, in addition, increased loading conditions and larger test networks.
In Subsection \ref{subApproximate}, we consider the first iteration of the proposed method as a linear approximant to the LF problem, and we compare its accuracy with existing approximate LF solution methods (LinDistFlow and \cite{Bolognani_EtAl-AllertonConf2015}).
Lastly, in Subsection \ref{BFSsubsec}, we illustrate similarities and differences with the well-known Backward-Forward Sweep method.

\subsection{Base Case Comparison}
\label{SubRiemComp}

\begin{table}
\renewcommand{\arraystretch}{1.0}
\renewcommand{\tabcolsep}{1mm}
\caption{Performance of Riemannian Optimization Methods (Base Case)}
\label{Tab:BaseCase}
\centering
\begin{tabular}{c|c|c|c|c|c|c}
    \hline
    \multirow{2}{*}{Network}&{Time (ms)}&\multicolumn{5}{c}{Riemannian Optimization Methods}\\
    \cline{3-7}
    & Iterations & GD(BFM) & P(BFM) & GD(QE) & N(QE) & P(QE) \\
    \hline
    \hline
    \multirow{2}{*}{13-node}& Time & 0.13$\times 10^3$ & 1.39 & 0.32$\times 10^3$ & 0.77 & 0.49\\
    & Iter \# & 162 & 3 & 1,420 & 3 & 3\\ 
    \hline
    \multirow{2}{*}{18-node}& Time & 21.6$\times10^{3}$ & 1.89 & 93$\times10^{3}$ & 2.7 & 0.55\\
    & Iter \# & 22,221 & 3 & 438,042 & 5 & 3\\ 
    \hline
    \multirow{2}{*}{22-node}& Time & 0.18$\times 10^3$ & 1.51 & 0.33$\times 10^3$ & 1.57 & 0.48\\
    & Iter \# & 151 & 2 & 1,288 & 3 & 2\\ 
    \hline
    \multirow{2}{*}{33-node}& Time & 3.8$\times 10^3$ & 3.36 & 54.2$\times 10^3$ & 3.54 & 0.96\\
    & Iter \# & 2,201 & 3 & 186,731 & 3 & 3\\
    \hline
    \multirow{2}{*}{37-node}& Time & 0.72$\times 10^3$ & 2.97 & 105.6$\times 10^3$ & 24.3 & 0.72\\
    & Iter \# & 363 & 2 & 213,478 & 3 & 2\\ 
    \hline
    \multirow{2}{*}{69-node}& Time & 4.38$\times 10^3$ & 7.17 & 3.5$\times 10^3$ & 14.3 & 1.98\\
    & Iter \# & 1,232 & 3 & 9,094 & 3 & 3\\ 
    \hline
    \multirow{2}{*}{85-node}& Time & 19.2$\times 10^3$ & 8.61 & 14.6$\times 10^3$ & 37.1 & 2.34\\
    & Iter \# & 4,383 & 3 & 31,415 & 4 & 3\\ 
    \hline
    \multirow{2}{*}{123-node}& Time & 23.5$\times 10^3$ & 14.83 & 35.7$\times 10^3$ & 125.7 & 3.29\\
    & Iter \# & 3,622 & 3 & 46,145 & 4 & 3\\ 
    \hline
    \multirow{2}{*}{141-node}& Time & 69.3$\times 10^3$ & 14.29 & 47.5$\times 10^3$ & 140 & 3.75\\
    & Iter \# & 9,709 & 3 & 64,469 & 4 & 3\\ 
    \hline
\end{tabular}
\end{table}

In Table \ref{Tab:BaseCase}, we compare the performance of the Riemannian optimization methods, for the Base Case, in terms of computation time and iterations.
The methods are denoted by: ``GD'' for Gradient Descent; ``N'' for Newton's; ``P'' for the proposed approximate Newton method; in parentheses we show the applicable manifold, BFM or QE.
We note that the times reported for Newton's method do not include the Hessian evaluation step, which was a time-consuming task that renders the performance of this method not acceptable in practice; for this reason we only included the results for the QE manifold mainly to compare the iterations it takes to converge with the proposed approximate Newton method.

The results show that the Riemannian GD converges after a considerably large number of iterations and requires significant computation time (in the order of seconds).
Conversely, the Riemannian Newton's method and the proposed approximate Newton method converge after only a few iterations (as expected), with the proposed method outperforming Newton's method in terms of computational effort in all test networks.
The results indicate that Newton's method (even excluding the Hessian evaluation) is slower than the proposed method by a factor that ranges from about 1.5 for the $13$-node to about 38 for the $123$-node test network.
In addition, we note that GD(BFM) takes much less iterations to converge, but each iteration is slower compared to GD(QE), by a factor ranging from 3.5 to 9.7, mainly due to the computational effort it takes to execute the retraction $\mathcal{R}^\textsc{BFM}$ compared to  $\mathcal{R}^\textsc{QE}_2$.\footnote{
We note that $\mathcal{R}^\textsc{QE}_2$ can be performed in parallel for each node/line, whereas $\mathcal{R}^\textsc{BFM}$ needs to be implemented in a forward sweep manner. 
Our implementation, however, did not take advantage of parallel processing, but exploited sparsity and vectorized calculation.}
For the same reason, although the P(BFM) and P(QE) converge at  the same number of iterations, the former is slower by a factor ranging from 2.8 to 4.5.\footnote{ 
We also tested Manopt variant of GD \cite{nocedal_EtAl-book2006} and Trust-Region (TR) \cite{Absil_EtAl-CompMath2007} methods on the BFM and QE manifolds.
Although Manopt GD converged in fewer iterations (about one order of magnitude less), each iteration was slower and the computation times were comparable to the ones reported in Table \ref{Tab:BaseCase}.
Manopt TR method converged in about 2-3 iterations in most cases, but the computation times were slower compared with the times reported for P(BFM) and P(GE) in Table \ref{Tab:BaseCase}, by about 2 orders of magnitude.}

\begin{figure}
  \centering
  \includegraphics[width=1\linewidth]{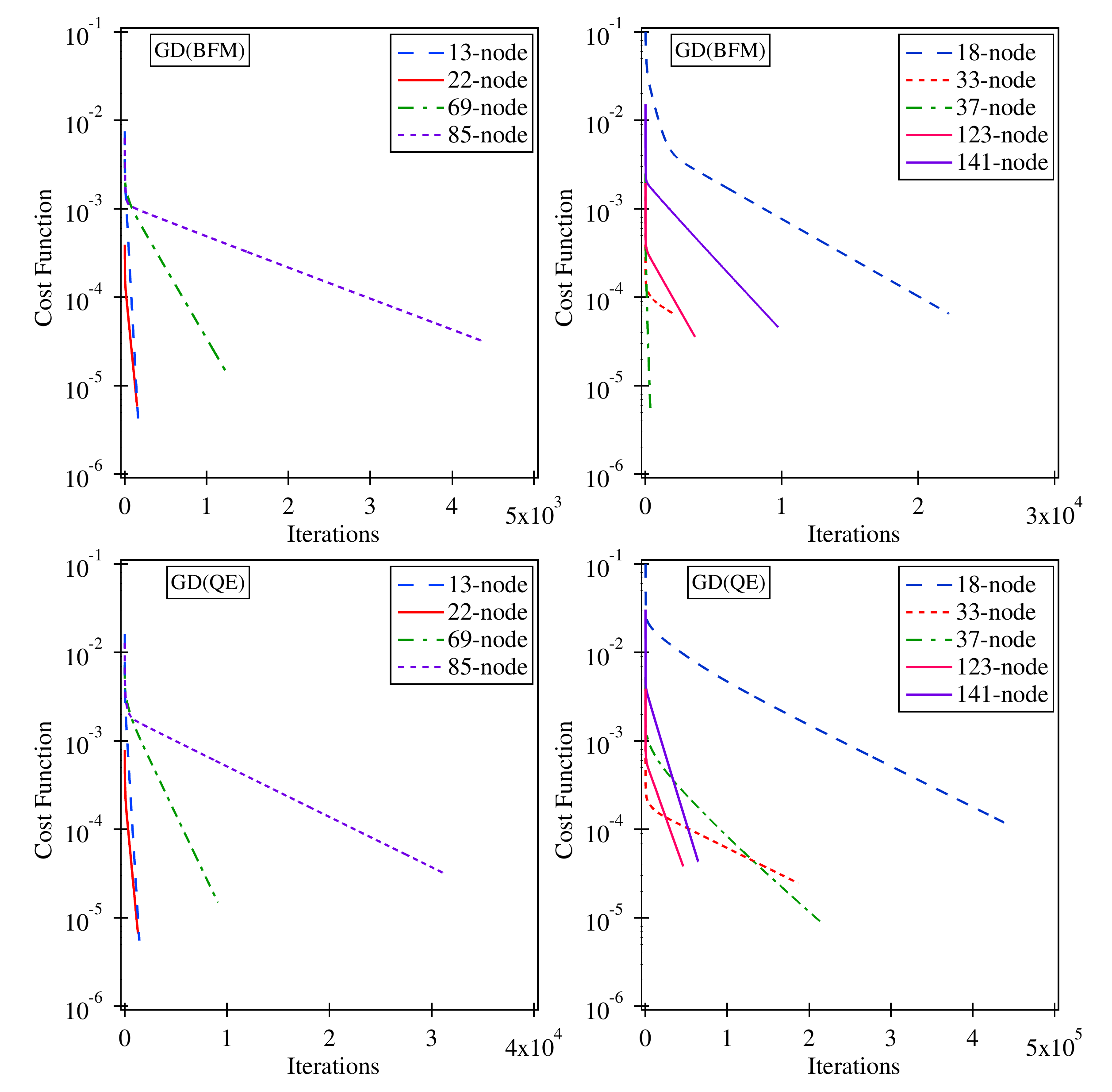}
  \caption{Trajectories (cost function value vs. iterations); Top:  GD(BFM); Bottom: GD(QE). The vertical axis is in logarithmic scale.} 
  \label{fig:GDCurve}
\end{figure}

\begin{figure}
  \centering
  \includegraphics[width=1\linewidth]{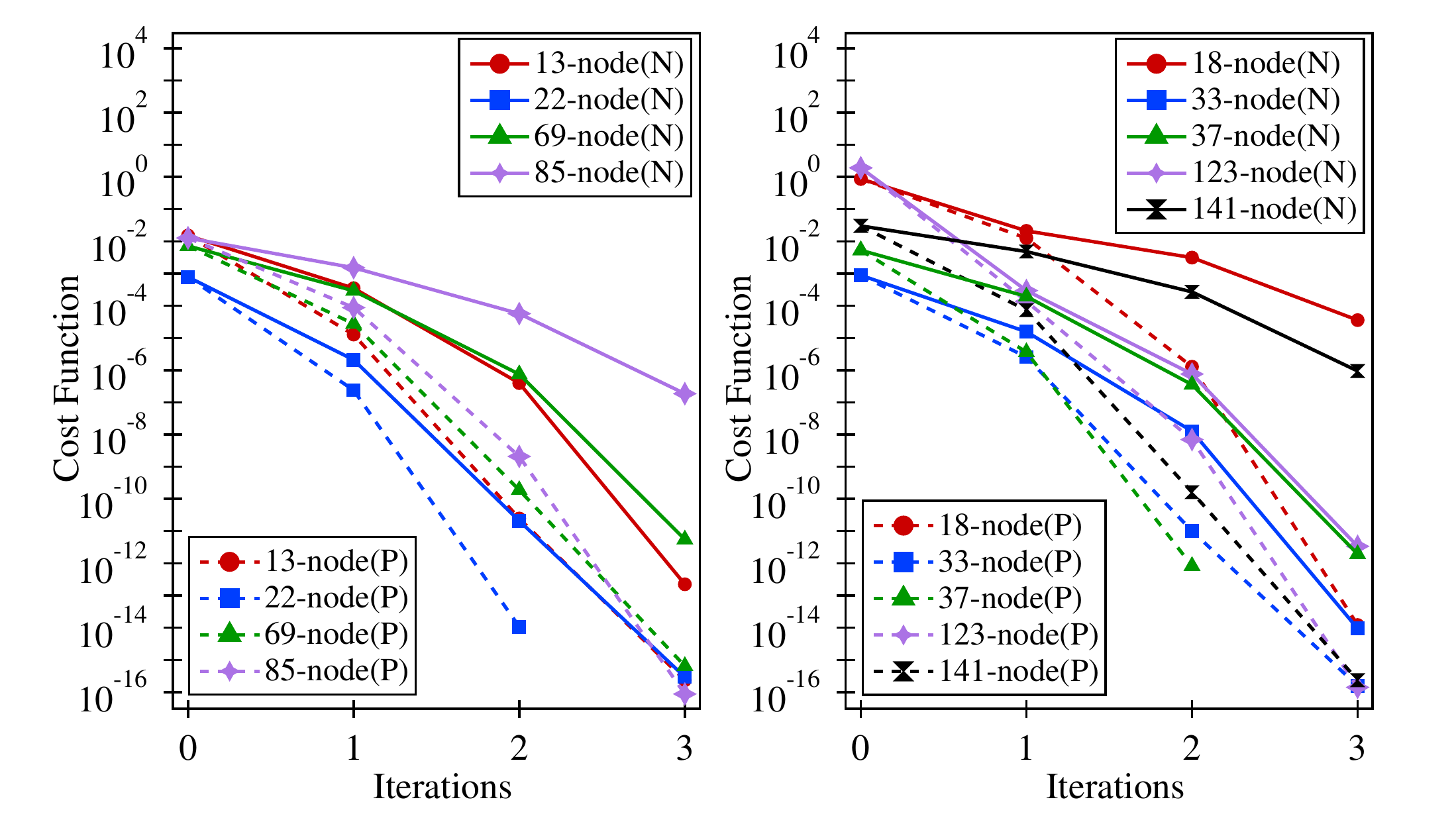}
  \caption{Trajectories (cost function value vs. iterations); N(QE): solid lines; P(QE): dashed lines. The vertical axis is in logarithmic scale.} 
  \label{fig:ObjFuncCount}
\end{figure}

Figure \ref{fig:GDCurve} illustrates the cost function trajectories for GD(BFM) (top) and GD(QE) (bottom), for all test networks.\
We observe that the cost function of GD on the BFM (QE) manifold decreases over consecutive iterations, and reaches a value of $10^{-4}$ within around 80 (760), 20K (455K), 20 (320), 420 (53K), 100 (91K), 700 (5.7K), 3K (22K), 2K (30K), and 7.7K (52K) iterations for the $13\text{-},$ $18\text{-},$ $22\text{-},$ $33\text{-},$ $37\text{-},$ $69\text{-},$ $85\text{-},$ $123\text{-},$ and $141\text{-}$node test networks, respectively.

Figure \ref{fig:ObjFuncCount} shows the cost function trajectories of the N(QE) and P(QE) for up to 3 iterations (since the latter method converges in at most 3 iterations for the employed test networks).
Notably, P(QE) curves are always below the corresponding N(QE) curves for all networks.
In most cases, the difference in terms of cost improvement in the first iteration is about one order of magnitude.
Also, note that whereas, for instance, the cost function of the $13$-node and $69$-node test networks is reduced to at least $0.01$ of its initial value in one or two iterations in Fig. \ref{fig:ObjFuncCount}, GD(QE) ---see Fig. \ref{fig:GDCurve}--- requires 653 and 6,260 iterations, respectively, to reach the same reduction.

\subsection{Comparison of Alternative Initializations and Retractions}
\label{SubCompInitRetr}

% In this subsection, we evaluate the performance of alternative initializations (\textsc{Flat}, \textsc{Warm}) and retractions ($\mathcal{R}^\textsc{QE}_{1}$, $\mathcal{R}^\textsc{QE}_{2}$).

\begin{table}
\renewcommand{\arraystretch}{1.0}
\caption{Performance of \textsc{Flat} on Riemannian Newton's and Proposed Approximate Newton Methods (Comparison with \textsc{Warm})}
\label{Tab:InitImpact}
\centering
\begin{tabular}{c|c|c|c|c}
    \hline
    \multirow{2}{*}{Network}&{Time (ms)}&\multicolumn{3}{c}{Riemannian Optimization Methods}\\
    \cline{3-5}
    & Iterations & P(BFM) & N(QE) & P(QE)\\
    \hline
    \hline
    \multirow{2}{*}{$13$-node}& Time & 2.21 (+0.82) &  1.26 (+0.49) & 0.61 (+0.12)\\
    & Iter \# & 4 (+1) & 5 (+2) & 4 (+1) \\ 
    \hline
    \multirow{2}{*}{$18$-node}& Time & 3.03 (+1.14) & 3.77 (+1.07) & 0.8 (+0.25)\\
    & Iter \# & 4 (+1) & 7 (+2) & 4 (+1)\\ 
    \hline
    \multirow{2}{*}{$22$-node}& Time & 2.63 (+1.12) & 2.03 (+0.46) & 0.55 (+0.07)\\
    & Iter \# & 3 (+1) & 4 (+1) & 3 (+1)\\ 
    \hline
    \multirow{2}{*}{$33$-node}& Time & 5.25 (+1.89) & 6.3 (+2.76) & 1.05 (+0.09)\\
    & Iter \# & 4 (+1) & 5 (+2) & 4 (+1)\\ 
    \hline
    \multirow{2}{*}{$37$-node}& Time & 4.4 (+1.43) & 24.6 (+0.3) & 0.82 (+0.1)\\
    & Iter \# & 3 (+1) & 5 (+2) & 3 (+1)\\ 
    \hline
    \multirow{2}{*}{$69$-node}& Time & 11 (+3.83) & 27.7 (+13.4) & 2.12 (+0.14) \\
    & Iter \# & 4 (+1) & 6 (+3) & 4 (+1)\\ 
    \hline
    \multirow{2}{*}{$85$-node}& Time & 13.49 (+4.88) & 59.8 (+22.7) & 2.49 (+0.15) \\
    & Iter \# & 4 (+1) & 6 (+1) & 4 (+1)\\ 
    \hline
    \multirow{2}{*}{$123$-node}& Time & 20.28 (+5.45) & 153 (+27.3) & 3.72 (+0.43) \\
    & Iter \# & 4 (+1) & 5 (+1) & 4 (+1)\\ 
    \hline
    \multirow{2}{*}{$141$-node}& Time & 22.39 (+8.1) & 182.9 (+42.9) & 3.96 (+0.21) \\
    & Iter \# & 4 (+1) & 5 (+1) & 4 (+1)\\ 
    \hline
\end{tabular}
\end{table}

%\subsubsection{\textsc{Flat} vs. \textsc{Warm}}
In Table \ref{Tab:InitImpact}, we evaluate the impact of \textsc{Flat} on the performance of Riemannian Newton's method and the proposed approximate Newton method.
The values in parentheses show the differences when \textsc{Warm} is employed, i.e., with the results of Table \ref{Tab:BaseCase}.  
The comparison verifies that \textsc{Flat}, which is considerably further to the optimal solution compared with \textsc{Warm}, performs worse in terms of both computation times and iterations (all differences are positive).
Notably, Corollary \ref{cor1} implies that P(QE), employing \textsc{Flat} and $\mathcal{R}^\textsc{QE}_{2}$, requires one more iteration compared with employing \textsc{Warm}.
The results in Table \ref{Tab:InitImpact} verify this outcome; the difference in time is actually the time required to execute the first iteration, which is in general lower than the average time per P(QE) iteration in Table \ref{Tab:BaseCase}.
In addition, P(QE) is affected much less (the computation time increase ranges from 0.07 ms for the 22-node to 0.43 ms for the 123-node test network) compared with Newton's method (whose computation time increase ranges from 0.46 ms to 27.3 ms for the same networks); hence, the differences observed in the Base Case (Table \ref{Tab:BaseCase}) become larger when employing \textsc{Flat}.
Newton's method is slower than the proposed method by a factor that ranges from 2 for the 13-node to 41 for the 123-node test network.
In general, the average time per iteration increases for P(BFM), which, compared with P(QE), becomes slower by a factor ranging from 3.6 to 5.6.

\begin{table}
\renewcommand{\arraystretch}{1.0}
\caption{Performance of $\mathcal{R}^\textsc{QE}_{1}$ on Proposed Approximate Newton Method (Comparison with $\mathcal{R}^\textsc{QE}_{2}$).}
\label{Tab:RetComp}
\centering
\begin{tabular}{c|c|c|c}
    \hline
    \multirow{2}{*}{Network}&{Time (ms)}&\multicolumn{2}{c}{Initializations}\\
    \cline{3-4}
    & Iterations & \textsc{Flat} & \textsc{Warm}\\
    \hline
    \hline
    \multirow{2}{*}{$13$-node}& Time &  0.73 (+0.12) & 0.58 (+0.09)\\
    & Iter \# &  4 (0) & 3 (0) \\ 
    \hline
    \multirow{2}{*}{$18$-node}& Time &  2.12 (+1.32) & 0.64 (+0.09)\\
    & Iter \# &  9 (+5) & 3 (0) \\ 
    \hline
    \multirow{2}{*}{$22$-node}& Time &  0.64 (+0.09) & 0.54 (+0.06)\\
    & Iter \# &  3 (0) & 2 (0) \\ 
    \hline
    \multirow{2}{*}{$33$-node}& Time &  1.17 (+0.12) & 1.05 (+0.09)\\
    & Iter \# &  4 (0) & 3 (0) \\ 
    \hline
    \multirow{2}{*}{$37$-node}& Time &  1.29 (+0.47) & 0.79 (+0.07)\\
    & Iter \# &  4 (+1) & 2 (0) \\ 
    \hline
    \multirow{2}{*}{$69$-node}& Time &  2.27 (+0.15) & 2.09 (+0.11)\\
    & Iter \# &  4 (0) & 3 (0) \\ 
    \hline
    \multirow{2}{*}{$85$-node}& Time &  2.61 (+0.12) & 2.44 (+0.10)\\
    & Iter \# &  4 (0) & 3 (0) \\ 
    \hline
    \multirow{2}{*}{$123$-node}& Time &  3.87 (+0.15) & 3.42 (+0.13)\\
    & Iter \# &  4 (0) & 3 (0) \\ 
    \hline
    \multirow{2}{*}{$141$-node}& Time &  5.45 (+1.49) & 3.88 (+0.13)\\
    & Iter \# &  5 (+1) & 3 (0) \\ 
    \hline
\end{tabular}
\end{table}

%\subsubsection{$\mathcal{R}^\textsc{QE}_{1}$ vs. $\mathcal{R}^\textsc{QE}_{2}$}
In Table \ref{Tab:RetComp}, we evaluate the impact of $\mathcal{R}^\textsc{QE}_{1}$ on the performance of the proposed approximate Newton method, which stands out as the most computationally efficient method, for both \textsc{Flat} and \textsc{Warm}.
The values in parentheses show the differences when $\mathcal{R}^\textsc{QE}_{2}$ is employed with either \textsc{Flat} or \textsc{Warm}.
The comparison suggests that when a closer initialization (\textsc{Warm}) is employed, both retractions perform well (yield the same number of iterations) with $\mathcal{R}^\textsc{QE}_{1}$ being slightly less computationally efficient ---the computation time increase is up to 0.13 ms, ranging from 3\% to 18\%.
The impact of $\mathcal{R}^\textsc{QE}_{1}$ combined with \textsc{Flat} is occasionally more severe (see, e.g., the 18-node test network, where the iterations increase by 5, and the computation time also increases by 165\%).

Lastly, we experimented with initial points selected intentionally to differ substantially from a reasonable LF solution.
Unreasonably low voltage and high current values were used.
We tested all methods on the $13$-node test network and observed convergence to a solution with low voltages that are not met in practical networks.
This was not a surprise.
LF equations \eqref{RealBalance}--\eqref{NonConvMan}, in general, admit multiple solutions, however, the solution with practical voltage magnitudes --- around 1 per-unit (p.u.) --- is unique \cite{Chiang_EtAl-CircSys1990}.
In fact, \cite{Chiang_EtAl-CircSys1990} shows a small example with two solutions; 
the realistic one and a low voltage one that is the type of solution we reached when we started from unreasonably low voltages.\footnote{
We also tried alternative Armijo parameters and in some cases managed to converge to the correct solution; however, convergence to a low voltage solution, in general, cannot be excluded if we start from a point that is close to that solution. Nevertheless, we note that both \textsc{Flat} and \textsc{Warm} initializations were close enough to the correct solution, as is shown in our results.}

%Initial Point
% W (voltage squared):[1.0000 0.6974 0.5495 0.4835 0.3956 0.2059 0.3536 0.2413 0.2680 0.3955 0.1140 0.0564 0.2982]
% l:[1.0000 1.4340 1.4340 1.4340 1.8199 2.0681 2.5275 2.5275 2.5275 3.7309 3.7309 2.5287]
% P: all 0.7071=1/sqrt(2)
% Q: all 0.7071=1/sqrt(2)

% Voltage Solution (P-BFM): [1.0000 0.3956 0.3888 0.1497 0.1584 0.3703 0.0089 0.1584 0.1468 0.1583 0.1391 0.1325 0.1339]
% Number of Iterations (P-BFM): 37

%Voltage solution (GD-BFM): [1.0000 0.5742 0.5697 0.5676 0.1232 0.5574 0.5654 0.1231 0.0727 0.1231 0.0242 0.0380 0.0832]
%Iters (GD-BFM): 16781

% Voltage Solution (P-QE with RET1): 1.0000 0.4437 0.3358 0.4352 0.1321 0.0239 0.4323 0.1321 0.1166 0.1320 0.1064 0.0970 0.0987
% Iters (P-QE): 14

% Voltage Solution (GD-QE):   1.0000, 0.3207, 0.2463, 0.1236, "0.0675i", 0.0353, 0.0112, "0.0647i", "0.0501i", 0.1133, "0.0675i", "0.0761i", 0.0741 } (NEGATIVE VOLTAGES)
% Iters (GD-QE): 361753

\subsection{Comparison with the Newton-Raphson Method}
\label{SubCompNR}
In this subsection, we compare the proposed Riemannian approximate Newton method with the traditional Newton-Raphson (NR) method (MATPOWER's implementation \cite{Zimmerman_EtAl-PowerSys2011}).

In Table \ref{Tab:NRComp}, we compare the performance of the NR method, using both \textsc{Flat} and \textsc{Warm} initializations, with P(QE). 
The values in parentheses show the differences; e.g., positive values in time suggest that NR is slower compared to P(QE).
The voltage phase angles required to warm start NR were recovered using \cite{Farivar_EtAl-PowerSys2013}.
The results show the same number of iterations and similar computational effort, implying that the proposed Riemannian approximate Newton method can achieve comparable performance with the NR method.

\begin{table}
\renewcommand{\arraystretch}{1.0}
\caption{Performance of Newton-Raphson Method (Comparison with the Proposed P(QE) Method).}
\label{Tab:NRComp}
\centering
\begin{tabular}{c|c|c|c}
    \hline
    \multirow{2}{*}{Network}&{Time (ms)}&\multicolumn{2}{c}{Initializations}\\
    \cline{3-4}
    & Iterations & \textsc{Flat} & \textsc{Warm}\\
    \hline
    \hline
    \multirow{2}{*}{$13$-node}& Time & 1.15 (+0.54) & 0.90 (+0.41)\\
    & Iter \# & 4 (0) & 3 (0)\\ 
    \hline
    \multirow{2}{*}{$18$-node}& Time & 1.23 (+0.43) & 0.98 (+0.43)\\
    & Iter \# & 4 (0) & 3 (0) \\ 
    \hline
    \multirow{2}{*}{$22$-node}& Time & 1.02 (+0.47) & 0.70 (+0.22)\\
    & Iter \# & 3 (0) & 2 (0) \\ 
    \hline
    \multirow{2}{*}{$33$-node}& Time & 1.54 (+0.49) & 1.19 (+0.23)\\
    & Iter \# & 4 (0) & 3 (0) \\ 
    \hline
    \multirow{2}{*}{$37$-node}& Time & 1.22 (+0.4) & 0.79 (+0.07)\\
    & Iter \# & 3 (0) & 2 (0) \\ 
    \hline
    \multirow{2}{*}{$69$-node}& Time &  2.30 (+0.18) & 1.79 (-0.19)\\
    & Iter \# & 4 (0) & 3 (0) \\ 
    \hline
    \multirow{2}{*}{$85$-node}& Time & 2.70  (+0.21) & 2.12 (-0.23)\\
    & Iter \# & 4 (0) & 3 (0) \\ 
    \hline
    \multirow{2}{*}{$123$-node}& Time & 3.45  (-0.27) & 2.59 (-0.7)\\
    & Iter \# & 4 (0) & 3 (0) \\ 
    \hline
    \multirow{2}{*}{$141$-node}& Time & 3.94  (-0.02) & 2.99 (-0.75)\\
    & Iter \# & 4 (0) & 3 (0) \\ 
    \hline
\end{tabular}
\end{table}

LF methods usually require more iterations and computational effort in higher loading conditions.
In Table \ref{Tab:LoadingComp}, we compare NR with P(QE), using \textsc{Warm} and $\mathcal{R}^\textsc{QE}_{2}$, under two increased loading scenarios, namely a medium and a high loading scenario.
Loading values are adjusted by a factor whose value is given in the first column of Table \ref{Tab:LoadingComp} for each network.
The values in parentheses show the differences with P(QE) obtained under the same loading conditions.
Positive (negative) differences declare worse (better) performance for NR compared with P(QE).
Indeed, the results in Table \ref{Tab:LoadingComp} show that the computation time increases by up to 1 ms for the medium and by up to 2 ms for the high loading scenario for the NR method (comparing with the values for the base loading scenario in Table \ref{Tab:NRComp}).
P(QE) exhibits an increase up to 0.44 ms for the medium and up to 2.66 ms for the high loading scenario (comparing with the values for the base loading scenario in Table \ref{Tab:BaseCase}).
Overall, the results for both methods under increased loading scenarios are comparable; 
the iterations remain practically the same (occasionally NR may need one more iteration) and the times are still in the order of a few milliseconds.
We also note that we tested P(BFM) with \textsc{Warm} initialization.
In almost all networks, the number of iterations was the same, but each P(BFM) iteration was slower compared with the P(QE) by a factor that ranged from 3.1 to 4.6, for both the medium and high loading scenarios, indicating a similar behavior with the base loading (reported in Table \ref{Tab:BaseCase}).

\begin{table}
\renewcommand{\arraystretch}{1.0}
\caption{Performance of Newton-Raphson under Increased Loading Scenarios (Comparison with the Proposed P(QE) Method)}
\label{Tab:LoadingComp}
\centering
\begin{tabular}{c|c|c|c}
    \hline
    \multirow{2}{*}{Network}&{Time (ms)}&\multicolumn{2}{c}{Loading Scenario}\\
    \cline{3-4}
    & Iterations & Medium & High\\
    \hline
    \hline
%    \multirow{2}{*}{13-node}& Time & 1.215 (+0.511) & 1.484 (+0.617)\\
    {$13$-node} & Time & 1.21 (+0.51) & 1.49 (+0.62)\\
    M:$\times2.5$; H:$\times3.5$ & Iter \# & 4 (0) & 5 (0)\\ 
    \hline
    %\multirow{2}{*}
    {$18$-node}& Time & 0.99 (+0.34) & 1.32 (+0.44)\\
    M:$\times1.5$; H:$\times2$& Iter \# & 3 (0) & 4 (0)\\ 
    \hline
    %\multirow{2}{*}
    {$22$-node}& Time & 1.38 (+0.64) & 1.73 (+0.52)\\
    M:$\times7$; H:$\times10$& Iter \# & 4 (+1) & 5 (0)\\ 
    \hline
	%\multirow{2}{*}
	{$33$-node}& Time & 1.59 (+0.59) & 1.99 (+0.34)\\
    M:$\times2.5$; H:$\times3.5$& Iter \# & 4 (+1) & 5 (0)\\ 
    \hline
    %\multirow{2}{*}
    {$37$-node}& Time & 1.19 (+0.11) & 1.97 (+0.18)\\
    M:$\times5$; H:$\times7.5$& Iter \# & 3 (0) & 5 (0)\\ 
    \hline
    %\multirow{2}{*}
    {$69$-node}& Time & 1.78 (-0.27) & 2.89 (-0.48)\\
    M:$\times2$; H:$\times3$& Iter \# & 3 (0) & 5 (0)\\ 
    \hline
    %\multirow{2}{*}
    {$85$-node}& Time & 2.08 (-0.32) & 4.11 (+0.15)\\
    M:$\times1.5$; H:$\times2.5$& Iter \# & 3 (0) & 6 (+1)\\ 
    \hline
    %\multirow{2}{*}
    {$123$-node}& Time & 3.34 (+0.08) & 5.04 (-0.35)\\
    M:$\times3$; H:$\times4.5$& Iter \# & 4 (+1) & 6 (+1)\\ 
    \hline
    %\multirow{2}{*}
    {$141$-node}& Time & 3.99 (+0.10) & 4.97 (-1.44)\\
    M:$\times3$; H:$\times4$& Iter \# & 4 (+1) & 5 (0)\\ 
    \hline
\end{tabular}
\end{table}

\begin{table}
\renewcommand{\arraystretch}{1.0}
\caption{Performance of Proposed P(QE) Method on Larger Networks (Newton-Raphson Performance in Parentheses).}
\label{Tab:LargeNetworks}
\centering
\begin{tabular}{c|c|c|c|c}
    \hline
    \multirow{2}{*}{Network} & Time (ms) &\multicolumn{3}{c}{Loading Scenario}\\
    \cline{3-5}
    & Iterations & Base & Medium & High \\
    \hline
    \hline
    & [Load] & [$\times1$] & [$\times2.5$] & [$\times 3.2]$ \\
     $906$-node   & Time & 22.5 (17.2) & 29.9 (23.2) & 30.2 (28.8) \\
     & Iter \#   & 3 (3) & 4 (4)& 4 (5)\\
    \hline
      & [Load] & $[\times1$]& $[\times1.5$] & [$\times1.9$]\\
    $2500$-node& Time & 72.4 (49.3)& 96.7 (65.6) & 121.0 (82.2)\\
     & Iter \# & 3 (3) & 4 (4)  & 5 (5) \\ 
    \hline
\end{tabular}
\end{table}

We further elaborate on the performance under several loading conditions for larger networks, a $906$-node European low voltage test network \cite{DistTestFeeders}, and a $2500$-node test network that is the single-phase equivalent of the $8500$-node network in \cite{DistTestFeeders}.
The results for the P(QE) method are presented in Table \ref{Tab:LargeNetworks}; the values in parentheses are the respective NR results. 
Although NR seems to perform better, the results are in the same order of magnitude (tens of milliseconds), which further enhances the argument that the proposed method achieves comparable performance with NR.
Lastly, we tested the performance of P(BFM).
The results indicated the same number of iterations and an increase in computation times by a factor of 6 and 7.6, for the $906$-node and the $2500$-node networks, respectively, under all loading conditions.

\subsection{Comparison of Approximate LF Solutions}
\label{subApproximate}

In this subsection, we consider approximate (not exact) LF solutions obtained by the first iteration of P(QE), employing \textsc{Warm} and $\mathcal{R}^\textsc{QE}_{2}$, referred to as the ``proposed approximant.''

\begin{figure}
  \centering
  \includegraphics[width=0.9\linewidth]{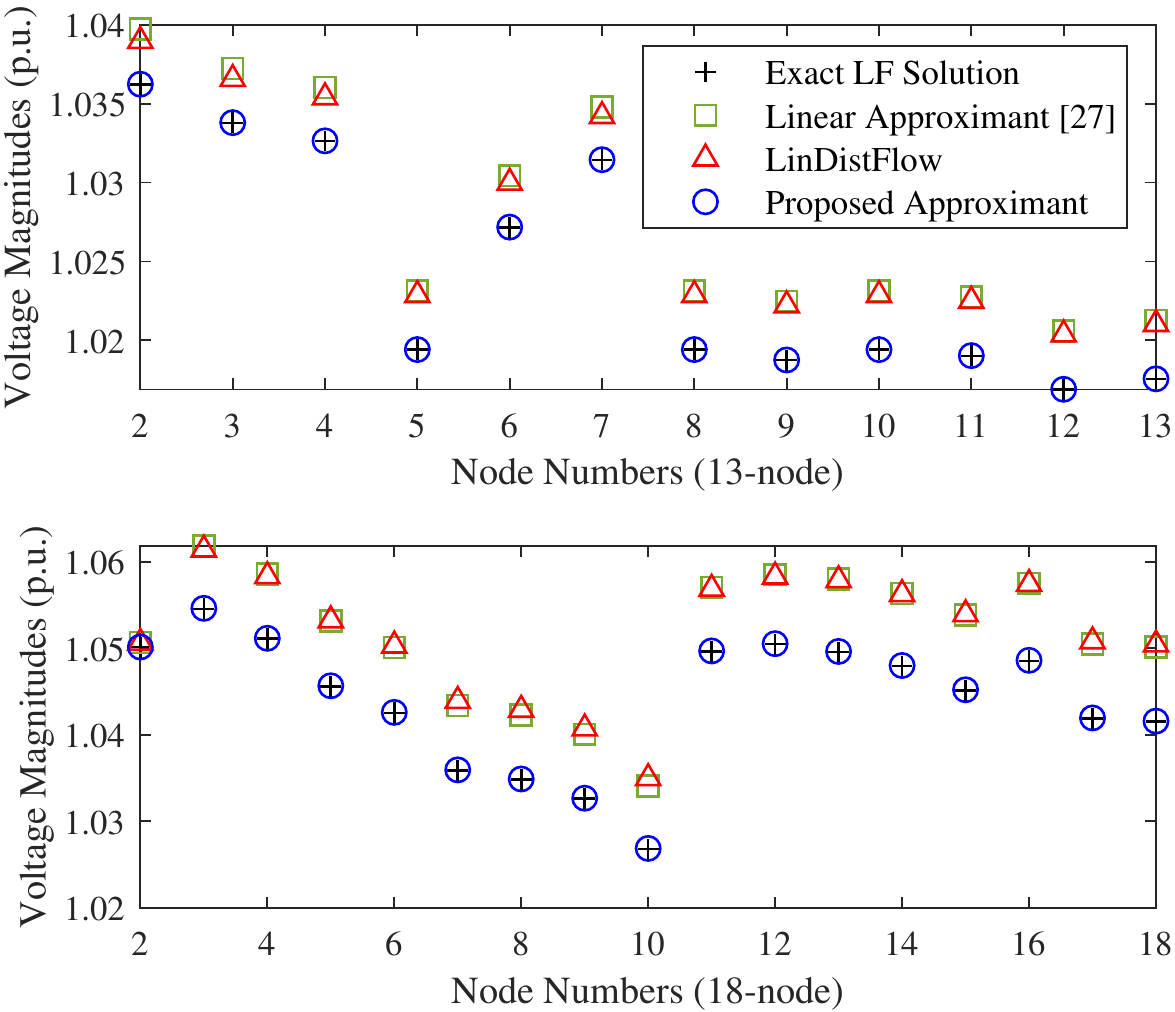}
  \caption{Comparison of approximate LF solutions on the $13$-node (upper) and the $18$-node (lower) test networks.}
  \label{fig:ApproxLF13and18}
\end{figure}
\begin{figure}
  \centering
  \includegraphics[width=1\linewidth]{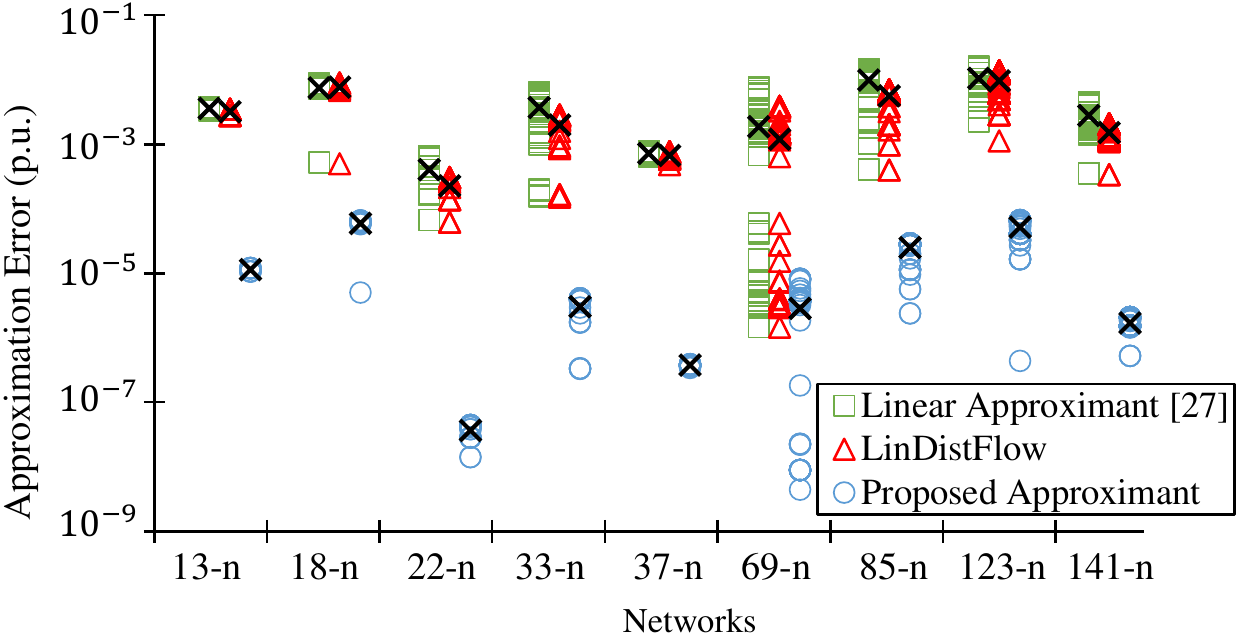}
  \caption{Voltage approximation error of the linear approximant in \cite{Bolognani_EtAl-AllertonConf2015}, the LinDistFlow solution and the proposed approximant. The ``x'' symbols represent the mean values.
  The vertical axis is in logarithmic scale.}
  \label{fig:ApproxLF}
\end{figure}

Figure \ref{fig:ApproxLF13and18} illustrates the voltage magnitudes obtained by the linear approximant in \cite{Bolognani_EtAl-AllertonConf2015}, the LinDistFlow solution and the proposed approximant ---including also the exact LF solution--- for the 13-node and 18-node test networks.
Note that, by definition of \textsc{Warm}, the LinDistFlow solution is the initial point of P(QE), hence, it is expected that the latter outperforms the former.
Figure \ref{fig:ApproxLF} illustrates the voltage approximation errors and their mean values for all methods.
The proposed approximant outperforms other approximants by at least two orders of magnitude on all test networks.
LinDistFlow achieves generally slightly better results compared with \cite{Bolognani_EtAl-AllertonConf2015} (recall, however, that the latter applies to meshed networks as well), which verifies the remarks in \cite{Bolognani_EtAl-AllertonConf2015} that LinDistFlow seems to improve the quality of their linear approximant (recall that the LinDistFlow solution is obtained by a non-linear change of coordinates in \cite{Bolognani_EtAl-AllertonConf2015} and assuming zero shunt admittances); however, we observe that \cite{Bolognani_EtAl-AllertonConf2015} yields a slightly better approximation than LinDistFlow in the 18-node test network. 

\subsection{Comparison with the Backward-Forward Sweep Method}
\label{BFSsubsec}
Last but not least, we consider the Backward-Forward Sweep (BFS) method, which has been proven to work well in radial networks.
Among several BFS variants, we discuss the approach proposed in \cite{Rajicic_EtAl-PowerDelivery1994}, as its forward sweep step (from the root to the leaf nodes) is identical to the $\mathcal{R}^\textsc{BFM}$ retraction.
At the backward sweep step (from leaf nodes to the root), BFS first moves from $\mathbf{x}_k$ to $\tilde{\mathbf{x}}_k$; then, at the forward sweep step, it applies $\mathcal{R}^\textsc{BFM}$ to find point $\tilde{\mathbf{x}}_{k+1}$ (on the BFM manifold).
The backward sweep  step ---which is described using the receiving-end power flows in \cite{Rajicic_EtAl-PowerDelivery1994}--- can be equivalently written using the sending-end flows, as follows:

Backward Sweep Step: $\tilde p_{j,k} = \bar p_j$, $\tilde q_{j,k} = \bar q_j$, $\tilde v_{j,k} = v_{j,k}$,
\begin{gather*}
\begin{split}
\tilde l_{j,k} = \Big[ &
\big(\sum_{j'\in \mathcal{J'}(j)}{\tilde P_{j',k}} + G_{j} \tilde v_{j,k} - \tilde p_j \big)^2 \\
&+ \big(\sum_{j' \in \mathcal{J'}(j)}{\tilde Q_{j',k}} - B_{j} \tilde v_{j,k} - \tilde q_j\big)^2 \Big] / \tilde v_{j,k}, 
\end{split}
\end{gather*}
and then $\tilde P_{j,k}$ and $\tilde Q_{j,k}$ are obtained from \eqref{RealBalance} and \eqref{ReactiveBalance} using $\tilde P_{j',k}, \tilde v_{j,k}$, $\tilde l_{j,k}$, and $\tilde p_{j,k}$, $\tilde q_{j,k}$.
Note the similarity with the direction finding step of P(BFM), where we also require $\tilde p_{j,k} = \bar p_j$, and $\tilde q_{j,k} = \bar q_j$ ---see the second set of  $\mathbf{w}_k+\boldsymbol{\eta}_k = \mathbf{\bar w}$ part in \eqref{wmis}.
Hence, the difference between BFS and P(BFM) is in the direction $\boldsymbol{\xi}_k$.
While both methods use the known nodal injections $\bar p_j$, and $\bar q_j$, BFS finds the direction by applying the backward sweep step, whereas P(BFM) requires the direction to be on the tangent space of the BFM manifold that is obtained by the solution of a linear system.
In other words, BFS also stays on the BFM manifold at each iteration but moves in a different direction.

We ran BFS on all test networks, and even though it generally took the same or a few more iterations to reach the solution (within the same tolerances), it was up to one order of magnitude faster.
Following our previous analysis, this result is not a surprise, as the P(BFM) requires the solution of a linear system in the direction finding step.
Figure \ref{fig:ObjFuncCountBFS} illustrates the cost function trajectories of P(BFM) and the evaluation of this function for the BFS method.
As expected, since the proposed method minimizes this function, P(BFM) trajectories are below the BFS ones; we also observe that BFS does not exhibit noticeable progress after the first and second iterations for the $37$-node and $123$-node test cases, respectively.

\begin{figure}
  \centering
  \includegraphics[width=1\linewidth]{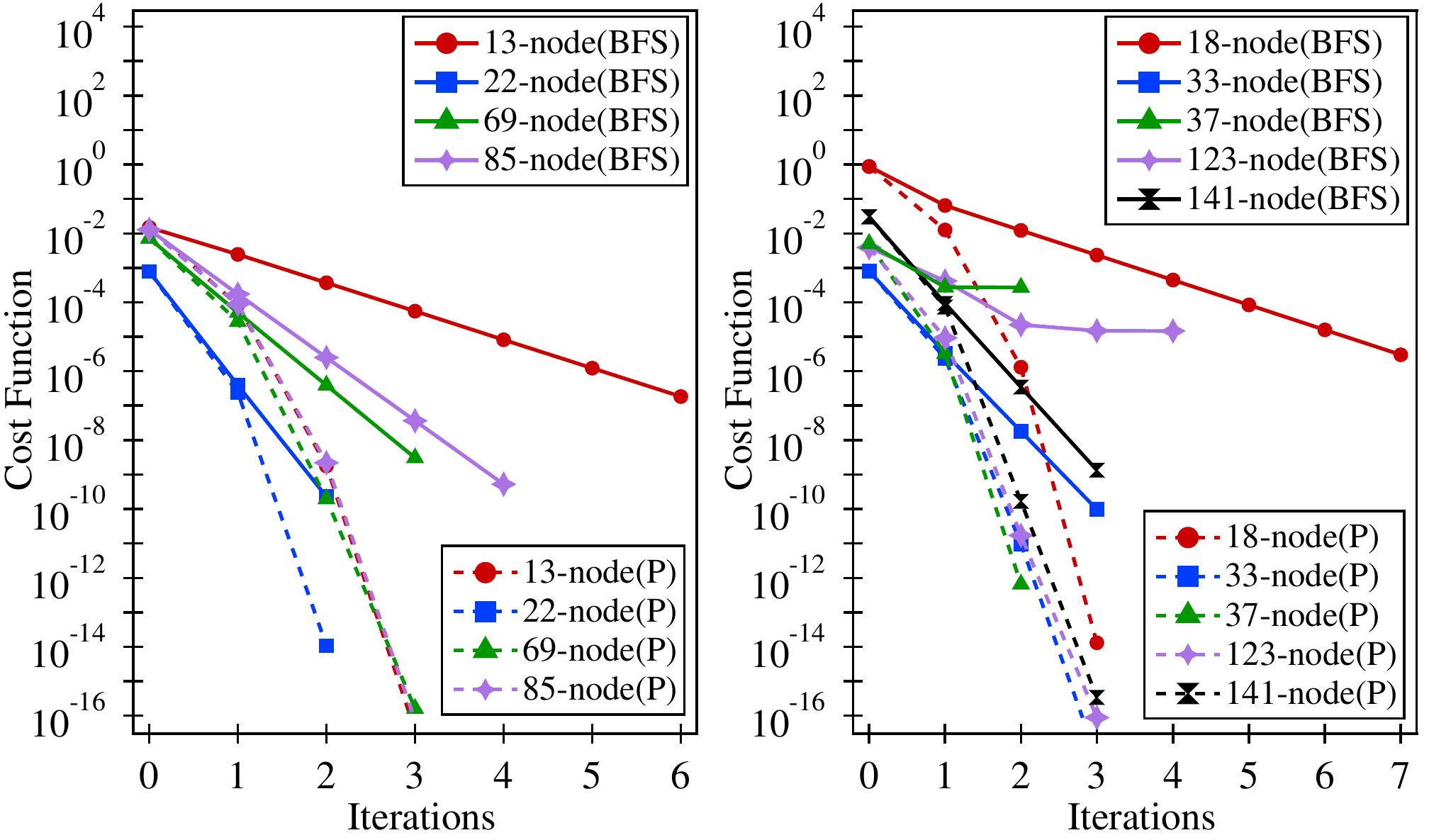}
  \caption{Trajectories (cost function value versus iterations); BFS: solid lines; P(BFM): dashed lines. The vertical axis is in logarithmic scale.} 
  \label{fig:ObjFuncCountBFS}
\end{figure}

Lastly, we should mention that the performance of BFS does not undermine the value of the proposed method.
The derived results and guarantees are promising for the application of Riemannian optimization to the more complicated OPF problem.
Furthermore, it has been observed that BFS may diverge under a high constant impedance loading condition (e.g., \cite{Araujo_EtAl-IJEPES2010} and \cite{Araujo_EtAl-IJEPES2018} report such cases on another BFS variant).
On an instance of the $123$-node test network (we replaced half of its constant power loads with their equivalent constant impedance, and applied a loading factor of 15.1), P(BFM) converged in 4 iterations with voltages spanning from about 0.7 to 1 p.u., whereas BFS exhibited oscillations in the voltage trajectory and diverged after 1K iterations.

\section{Conclusions and Future Research}
\label{SecConc}
In this paper, we introduced a novel Riemannian optimization approach to the LF problem in radial distribution networks, employing the branch flow model.
Our proposed method was shown to fall into the category of Riemannian approximate Newton methods, and guarantee descent at each iteration while maintaining a local superlinear convergence rate.
Extensive numerical results illustrated that the proposed method outperforms other Riemannian optimization methods, namely the Riemannian Gradient Descent and the Riemannian Newton's method, and that it achieves comparable performance with the traditional Newton-Raphson method.
Also, we observed that the first iteration of the proposed method yields an approximate LF solution that is of higher quality (by at least two orders of magnitude) compared with other linear LF approximants.
Lastly, we presented an interesting comparison with the well-known backward-forward sweep method, illustrating that while both methods essentially stay on the manifold, they move along different directions.

Our future research considers two directions. Firstly, we plan to extend the proposed Riemannian LF solution method to a general multi-phase power distribution network.
In a multi-phase setting, the presence of mutual admittances between phases adds several degrees of complexity to the BFM and requires identifying new valid retractions.
Secondly, and perhaps most importantly, we plan to address the more challenging OPF problem.
In an OPF setting, the BFM manifold combined with operational constraints yields a non-smooth manifold that requires approaches such as the Riemannian augmented Lagrangian or exact penalty methods introduced in \cite{Liu_EtAl-Springer2019}.

% or


\appendix[Omitted Proofs]

\subsection{Proof of Lemma \ref{LemmaRetBFM} }
\label{AppRetBFM}
The proof for the centering condition is straightforward, since a zero tangent vector yields the tangent point, i.e., the current iterate $\mathbf{x}_k$.
For the local rigidity condition, without loss of generality, we assume that the stepsize $\alpha_k$ is 1. 
We first consider the retraction associated with the variable $l_{k+1}$.
Taking the derivative at $t=0$ yields:
\begin{equation}
    \frac{\mathrm d}{\mathrm d t}\mathcal{R}^\textsc{BFM}_{\mathbf{x}_k , l}(t\boldsymbol{\xi}_k)|_{t=0}=\frac{2(\zeta_{k}^{P}P_{k} + \zeta_{k}^{Q}Q_{k})v_{i,k} - \zeta_{k}^{v}(P_k^2 + Q_k^2)}{v_{i,k}^2}, \label{Ret2LRigidCond}
\end{equation}
where $\zeta_{k}^{P}, \zeta_{k}^{Q}, \zeta_{k}^{v}$ denote the elements of $\boldsymbol{\xi}_k$ associated with $P_k$, $Q_k$, and $v_{i,k}$, respectively.
Using the fact that $\mathbf{x}_k$ lies on the manifold, we have:
\begin{equation}
    P_{k}^{2}+Q_{k}^{2}=l_{k}v_{i,k}, \label{ManCurIter}
\end{equation}
whose tangent space is characterized as:
\begin{equation}
    2(\zeta_{k}^{P}P_{k}+\zeta_{k}^{Q}Q_{k})=\zeta_{k}^{v}l_{k}+\zeta_{k}^{l}v_{i,k}, \label{TangCurIter}
\end{equation}
\eqref{Ret2LRigidCond} yields $\frac{\mathrm d}{\mathrm d t}\mathcal{R}^\textsc{BFM}_{\mathbf{x}_k , l}(t\boldsymbol{\xi}_k)|_{t=0}=\zeta_{k}^{l}$, where $\zeta_{k}^{l}$ denotes the element of $\boldsymbol{\xi}_k$ associated with $l_k$.
The proof for variables $P_{k+1}$, $Q_{k+1}$, $v_{j,k+1}$, $p_{k+1}$ and $q_{k+1}$ is straightforward, since the retraction mappings are linear.

\subsection{Proof of Lemma \ref{LemmaRetQE1}}
\label{AppRetQE1}
The proof for the centering condition is straightforward.
For the local rigidity condition, assuming a stepsize $\alpha_k$ equal to 1, we first consider the proof for $l_{k+1}$.
After some algebraic manipulations, the derivative at $t=0$ yields:
\begin{gather} 
 \frac{\mathrm d}{\mathrm d t} \mathcal{R}^\textsc{QE}_{1,\mathbf{u}_k , l}(t\boldsymbol{\zeta}_k)|_{t=0} = 
 \frac{\zeta_k^{v} l_{k}}{v_{i,k}} +
 \frac{2  (\zeta_{k}^{P}P_{k}+\zeta_{k}^{Q}Q_{k})(v_{i,k}-l_{k})}{v_{i,k}(l_{k}+v_{i,k})} \nonumber\\ 
+ \frac{2 (P_{k}^{2}+ Q_{k}^{2}) (\zeta_{k}^{l}-\zeta_{k}^{v})}{v_{i,k}(l_{k}+v_{i,k})}.
\label{DiffL}
\end{gather}
Using \eqref{ManCurIter} and \eqref{TangCurIter}, \eqref{DiffL} yields $\frac{\mathrm d}{\mathrm d t}\mathcal{R}^\textsc{QE1}_{\mathbf{u}_k , l}(t\boldsymbol{\zeta}_k)|_{t=0}=\zeta_{k}^{l}$.
We then consider the retraction for variable $P_{k+1}$, which can be written as:
\begin{equation} \label{RxP}
\mathcal{R}^\textsc{QE}_{1,\mathbf{u}_k , P}(\boldsymbol{\zeta}_k) = \frac{2\Tilde{P}_k}{D_k+\Tilde{l}_k-\Tilde{v}_{i,k}} \mathcal{R}^\textsc{QE}_{1,\mathbf{u}_k , l}(\boldsymbol{\zeta}_k).
\end{equation}
Employing the product rule for derivatives, and using \eqref{ManCurIter}--\eqref{TangCurIter}, the centering condition, and the the local rigidity condition for $l_{k+1}$, \eqref{RxP} yields:
\begin{equation}
 \frac{\mathrm d}{\mathrm d t}\mathcal{R}^\textsc{QE}_{1,\mathbf{u}_k , P}(t\boldsymbol{\xi}_k)|_{t=0}= \frac{\zeta^{P}_k l_k - P_k  \zeta^{l}_k}{l^2_{k}} l_k + \frac{P_k}{l_k} \zeta_k^l 
= \zeta_{k}^{P}. \nonumber   
\end{equation}
The proof for $Q_{k+1}$ is similar and hence omitted.
The proof for $v_{j,k+1}$ is straightforward.

\subsection{Proof of Lemma \ref{LemmaRetQE2} }
\label{AppRetQE2}
The proof for the centering condition is straightforward.
For the local rigidity condition, the proof for variables $P_{k+1}$, $Q_{k+1}$, and $v_{j,k+1}$, is also straightforward.
The proof for variable $l_{k+1}$ follows the respective proof of Lemma \ref{LemmaRetBFM}.

\subsection{Proof of Proposition \ref{Prop1}}
\label{ProofProp1}

\begin{lemma}
\label{PropositionDescent}
The direction sequence $\{\boldsymbol{\xi}_k\}$ generated from the solution of \eqref{wmis} is gradient-related (\cite[Def. 4.2.1]{Abs_EtAl-Book2008}).
\end{lemma}
\begin{proof}
It suffices to show that the pair $(\mathbf{x}_k,\boldsymbol{\xi}_k)$ satisfies \eqref{descentDirection}, i.e., $\langle\mathrm{grad} f_\textrm{BFM}(\mathbf{x}_k), \boldsymbol{\xi}_k\rangle < 0$.
Expanding the dot product and using \eqref{OurRiemGradEntirePF}, the \emph{lhs} of \eqref{descentDirection} can be written as:
\begin{gather}
    \langle\mathrm{grad} f_\textrm{BFM}(\mathbf{x}_{k}), \boldsymbol{\xi}_k\rangle
    = 2\boldsymbol{\xi}_k^{T}\mathbf{\Pi}_{\mathbf{x}_{k}} \begin{pmatrix} \mathbf{0}_{4J\times1} \\ \mathbf{w}_k-\mathbf{\bar w}
  \end{pmatrix}.
    \label{descent1}
\end{gather}
Using the fact that $\mathbf{\Pi}_{\mathbf{x}_{k}}$ is an orthogonal projection matrix satisfying $\mathbf{\Pi}_{\mathbf{x}_{k}} = \mathbf{\Pi}_{\mathbf{x}_{k}}^T$ \cite{Strang-Book2006}, and the fact that $\boldsymbol{\xi}_k$ lies on the tangent space, and hence, $\mathbf{\Pi}_{\mathbf{x}_{k}}  \boldsymbol{\xi}_k = \boldsymbol{\xi}_k$ (intuitively, the projection of a vector that is already on the tangent space should be itself), we get:
\begin{equation} \label{xi_aux}
    \boldsymbol{\xi}_k^{T}\mathbf{\Pi}_{\mathbf{x}_{k}}
    =\big(\mathbf{\Pi}_{\mathbf{x}_{k}}^T  \boldsymbol{\xi}_k   \big)^{T}
    = \big(\mathbf{\Pi}_{\mathbf{x}_{k}}  \boldsymbol{\xi}_k   \big)^{T}
    = \boldsymbol{\xi}_k^{T}.
\end{equation}
Using \eqref{wmis} and \eqref{xi_aux}, \eqref{descent1} yields: 
\begin{equation}
    \langle\mathrm{grad} f_\textrm{BFM}(\mathbf{x}_{k}), \boldsymbol{\xi}_k\rangle
    = -2 \boldsymbol{\eta}_k^{T}  
    \boldsymbol{\eta}_k
    = -2\| \boldsymbol{\eta}_k\|^{2} < 0,
    \label{DescentTang1}
\end{equation}
where $\boldsymbol{\eta}_k\neq \mathbf{0}$, otherwise \eqref{wmis} implies that $\mathbf{x}_k$, a point on the manifold, has achieved the global minimum of \eqref{LF_as_Opt}, hence the optimal solution is reached.
\end{proof}

Lemma \ref{PropositionDescent}, $\mathcal{R}^\textsc{BFM}$ ---which from Lemma \ref{LemmaRetBFM} satisfies the retraction definition--- and the Armijo rule, guarantee descent at each iteration; hence,
from \cite[Thm. 4.3.1]{Abs_EtAl-Book2008}, every accumulation (limit) point of $\{\mathbf{x}_k\}$, denoted by $\{\mathbf{x}_*\}$, is a critical (stationary) point of the cost function $f_\textrm{BFM}$.

We then show that \eqref{wmis} can be written in the form of \eqref{ApproximateJacob}, using an approximate Hessian (Jacobian), hence it falls into the category of Riemannian approximate Newton methods.
Recall that the Jacobian matrix in \eqref{ApproximateJacob} is the Riemannian Hessian, i.e., $\mathbf{J}( \mathbf{x}_k ) := \mathrm{hess} f ( \mathbf{x}_k )$.
Rearranging the terms in \eqref{wmis}, appending both sides to a $4J\times1$ vector of zeros, and multiplying with $2\mathbf{\Pi}_{\mathbf{x}_{k}}$, we get:
\begin{equation}
     2\mathbf{\Pi}_{\mathbf{x}_{k}} \begin{pmatrix}
    \mathbf{0}_{4J\times4J} \quad \mathbf{0}_{4J\times2J} \\
    \mathbf{0}_{2J\times4J} \quad \mathbf{I}_{2J\times2J} \\
    \end{pmatrix} \boldsymbol{\xi}_k = - 2\mathbf{\Pi}_{\mathbf{x}_{k}} \begin{pmatrix} \mathbf{0}_{4J\times1} \\ \mathbf{w}_k-\mathbf{\bar w}
  \end{pmatrix} , \label{Ax=bMod2}
\end{equation}
where the \emph{rhs} is the Riemannian gradient given by \eqref{OurRiemGradEntirePF}.
Employing the Riemannian Hessian given by \eqref{RiemHessProofEntirePFMan}, and using the property that $\mathbf{\Pi}_{\mathbf{x}_{k}}^2 = \mathbf{\Pi}_{\mathbf{x}_{k}}$ \cite{Strang-Book2006}, we can write \eqref{Ax=bMod2} in the form of  \eqref{ApproximateJacob}, with $\mathbf{E}_k = - \mathbf{\Pi}_{\mathbf{x}_k} \mathbf{{C}}_{\mathbf{x}_k}$.
The local superlinear convergence rate is shown in \cite[Thm. 8.2.1]{Abs_EtAl-Book2008}, provided that $ \|\mathbf{E}_{k}\|_2 \leq \gamma_1 \|\mathrm{grad}f_\textrm{BFM}(\mathbf{x}_{k}) \|_2,$
for some constant $\gamma_1$.
Using \eqref{EqLixkEntirePFMan}, the $n$-th column of square matrix $\mathbf{E}_{k}$, is expressed as $\mathbf{E}_{k,n} = - 2 \boldsymbol{\Pi}_{x_{k}} \boldsymbol{\Gamma}_{n,\mathbf{x}_k}^{T}  \begin{pmatrix} \mathbf{0}_{4J\times1} \\ \mathbf{w}_k-\mathbf{\bar w}
  \end{pmatrix}$,
and hence, using \eqref{OurRiemGradEntirePF}, we get $
    \|\mathbf{E}_{k}\|_2\leq\sum_{n}\|\mathbf{E}_{k,n}\|_2
    \leq \sum_{n} \|\boldsymbol{\Gamma}_{n,\mathbf{x}_k}\|_2
    \|\mathrm{grad}f_\textrm{BFM}(\mathbf{x}_{k})\|_2,$
with $ \gamma_1 = \sum_n \|\boldsymbol{\Gamma}_{n,\mathbf{x}_k}\|_2$.

\subsection{Proof of Proposition \ref{Prop2}}
\label{ProofProp2}

\begin{lemma}
\label{PropositionDescent2}
The direction sequence $\{\boldsymbol{\zeta}_k\}$ generated from the solution of \eqref{QE_xiTang} is gradient-related (\cite[Def. 4.2.1]{Abs_EtAl-Book2008}).
\end{lemma}
\begin{proof}
The search direction $\boldsymbol{\zeta}_k$ in \eqref{QE_xiTang} satisfies:
\begin{align*}
\langle\mathrm{grad} f_\textrm{QE}(\mathbf{u}_{k}), \boldsymbol{\zeta}_k\rangle
    & = 2\boldsymbol{\zeta}_k^{T}\mathbf{\Pi}_{\mathbf{u}_{k}} \mathbf{A}^T (\mathbf{A} \mathbf{u}_{k}-\mathbf{b})\\
    & = -2 \boldsymbol{\zeta}_k^{T}\mathbf{\Pi}_{\mathbf{u}_{k}} \mathbf{A}^T \mathbf{A} \boldsymbol{\zeta}_k = -2 \boldsymbol{\zeta}_k^{T} \mathbf{A}^T \mathbf{A} \boldsymbol{\zeta}_k \nonumber \\
    & = -2 (\mathbf{A} \boldsymbol{\zeta}_k)^{T}  
    (\mathbf{A} \boldsymbol{\zeta}_k) 
    = -2\|\mathbf{A} \boldsymbol{\zeta}_k\|^{2} < 0,\nonumber    
\end{align*}
where we used \eqref{xi_aux} to write $\boldsymbol{\zeta}_k^{T}\mathbf{\Pi}_{\mathbf{u}_{k}}=\boldsymbol{\zeta}_k^{T}$.
Note that $\mathbf{A} \boldsymbol{\zeta}_k \neq 0$, otherwise we should have reached the optimal solution.
% +++The sequence $\| \boldsymbol{\zeta}_k \|$ is bounded, as    $\|\mathrm{grad}f_\textrm{QE}(\mathbf{u}_{k})\|^2 = \| 2 \mathbf{\Pi}_{\mathbf{u}_{k}} \mathbf{A}^T \mathbf{A} \boldsymbol{\zeta}_k \|^2  \geq 4\underline{\lambda} \| \boldsymbol{\zeta}_k \|^2$, where $\underline{\lambda} > 0$ is the minimum Eigenvalue of the positive definite matrix $LL^T$ where $L = \mathbf{\Pi}_{\mathbf{u}_{k}} \mathbf{A}^T \mathbf{A}$ \cite{Strang-Book2006}. 
\end{proof}
Similarly to the proof of Proposition \ref{Prop1}, Lemma \ref{PropositionDescent2}, Armijo rule, Lemmas \ref{LemmaRetQE1} and \ref{LemmaRetQE2}, and \cite[Thm. 4.3.1]{Abs_EtAl-Book2008} guarantee descent and that $\{\mathbf{x}_*\}$ is a critical point of the cost function $f_\textrm{QE}$.
Then, multiplying both sides of \eqref{QE_xiTang} with $2\mathbf{\Pi}_{\mathbf{u}_{k}}\mathbf{A}^T$ yields $2\mathbf{\Pi}_{\mathbf{u}_{k}}\mathbf{A}^T \mathbf{A} \boldsymbol{\zeta}_k = - 2\mathbf{\Pi}_{\mathbf{u}_{k}}\mathbf{A}^T \big(\mathbf{A} \mathbf{u}_k - \mathbf{b}\big)$, which, using  \eqref{RiemHessProof}, can be written in the form of \eqref{ApproximateJacob}, with $\mathbf{E}_k = - \mathbf{\Pi}_{\mathbf{u}_k} \mathbf{{L}}_{\mathbf{u}_k}$.
Lastly, local superlinear convergence rate from \cite[Thm. 8.2.1]{Abs_EtAl-Book2008} holds for $\gamma_2 = \sum_n \|\boldsymbol{{\Lambda}}_{n,\mathbf{u}_k}\|_2$, since, using \eqref{EqLixk}, the $n$-th column of $\mathbf{E}_{k}$, is expressed as $\mathbf{E}_{k,n} = - 2 \mathbf{\Pi}_{\mathbf{u}_{k}} \boldsymbol{\Lambda}_{n,\mathbf{u}_k}^{T}  \mathbf{A}^T (\mathbf{A} \mathbf{u}_{k}-\mathbf{b})$, and using \eqref{OurRiemGrad}, we have $\|\mathbf{E}_{k}\|_2 \leq \sum_{n} \|\mathbf{E}_{k,n}\|_2
        \leq \sum_n 
        \|\boldsymbol{\Lambda}_{n,\mathbf{u}_k}\|_2 \|\mathrm{grad}f_\textrm{QE}(\mathbf{u}_{k}) \|_2.$

\subsection{Proof of Corollary \ref{cor1}}
\label{Proofcor1}

Employing \textsc{Flat}, and assuming, without loss of generality that the slack node voltage is equal to 1, $\mathbf{u}_0$ is given by $\mathbf{P}_0=\mathbf{0}$, $\mathbf{Q}_0=\mathbf{0}$, $\mathbf{l}_0=\mathbf{0}$, and $\mathbf{v}_0=\mathbf{1}$.
Hence, \eqref{QE_xiTang} yields $\boldsymbol{\zeta}_0 \in \mathcal{T}_{\mathbf{u}_0}\mathcal{M}_\textrm{QE}$, requiring that $\boldsymbol{\zeta}_{0}^{l} = \mathbf{0}$, where $\boldsymbol{\zeta}_{0}^{l}$ is the element of $\boldsymbol{\zeta}_0$ associated with the squared current magnitude.
Also, \eqref{QE_xiTang} yields $\mathbf{A} \Tilde{\mathbf{u}}_0 = \mathbf{b}$, representing the simplified DistFlow equations \eqref{DistFlow1}--\eqref{DistFlow3}
for $\Tilde{\boldsymbol{u}}_0 = \boldsymbol{u}_0 +\boldsymbol{\zeta}_0$ (more precisely for $\Tilde{\mathbf{P}}_0$, $\Tilde{\mathbf{Q}}_0$ and $\Tilde{\mathbf{v}}_0$, since $\Tilde{\mathbf{l}}_0 = \mathbf{l}_0 + \boldsymbol{\zeta}_{0}^{l} = \mathbf{0}$).
Retraction $\mathcal{R}^\textsc{QE}_{2}$ then derives the values of the current using \eqref{RetQE2}.

%\section*{Acknowledgment}

\ifCLASSOPTIONcaptionsoff
  \newpage
\fi



%\bibliographystyle{IEEEtran}

%\begin{thebibliography}{1}
%\end{thebibliography}

\bibliographystyle{IEEEtran}
%\IEEEtriggeratref{}
\bibliography{IEEEabrv,Ref}


%\begin{IEEEbiography}{Majid Heidarifar}
%\end{IEEEbiography}


%\newpage



%\vfill


\end{document}